\documentclass[11pt]{amsart}
 
\usepackage{tikz}
\usetikzlibrary{cd}
\usetikzlibrary{decorations.pathreplacing,decorations.markings}

\usepackage{amsfonts,amssymb,amsmath}

\usepackage{pinlabel,url}
\usepackage[all]{xy}
\usepackage{tikz}
\usepackage{graphicx}
\usepackage{xcolor}
\usepackage{color}
\usepackage[normalem]{ulem}
\usepackage[colorlinks,citecolor=blue]{hyperref}
\usepackage[sorting=nty,style=alphabetic,url=false,maxnames=100]{biblatex}
\usepackage[nameinlink]{cleveref}

\theoremstyle{theorem}
\newtheorem{theorem}{Theorem}[section]

\newtheorem{lemma}[theorem]{Lemma}
\newtheorem{proposition}[theorem]{Proposition}
\newtheorem{corollary}[theorem]{Corollary}
\newtheorem{claim}[theorem]{Claim}

\theoremstyle{definition}

\newtheorem*{remark}{Remark}

\newtheorem{definition}[theorem]{Definition}

\newtheorem{question}[theorem]{Question}

\newtheorem{innercustomthm}{Theorem}
\newenvironment{customthm}[1]
  {\renewcommand\theinnercustomthm{#1}\innercustomthm}
  {\endinnercustomthm}

\newtheorem{innercustomcor}{Corollary}
\newenvironment{customcor}[1]
  {\renewcommand\theinnercustomcor{#1}\innercustomcor}
  {\endinnercustomcor}

\numberwithin{equation}{section}

\newcommand{\Mod}{{\rm Mod}}
\newcommand{\PMod}{{\rm PMod}}

\newcommand{\Map}{{\rm Map}}
\newcommand{\PMap}{{\rm PMap}}

\newcommand{\T}{\mathcal{T}}

\newcommand{\CL}{\mathcal{L}}

\newcommand{\hyp}{{\mathbb{H}}}
\newcommand{\SH}{{\mathcal{SH}}}
\newcommand{\id}{{\rm id}}

\newcommand{\cF}{{\mathcal F}}

\newcommand{\lcm}{\rm{lcm}}
\newcommand{\N}{\mathbb{N}}
\newcommand{\esup}{{\rm esssup}}
\newcommand{\R}{\mathbb{R}}

\newcommand{\Bel}{{\rm Bel}}
\newcommand{\SHo}{{{\mathcal H}}}
\newcommand{\PSHo}{\mathcal{PH}}

\newcommand{\Imag}{{\rm{Im}}}
\newcommand{\E}{\mathcal{E}}

\newcommand{\cR}{\mathcal{R}}
\newcommand{\inte}{\operatorname{int}}

\renewcommand{\Im}{\operatorname{Im}}

\newcounter{ccomments}

\newcounter{gcomments}

\newcounter{jcomments}

\title[WP Geometry and end-periodic surfaces]{Asymptotically conformal and asymptotically rigid mapping class groups}

\author[J. Aramayona]{Javier Aramayona}
\address{Javier Aramayona: Instituto de Ciencias Matem\'aticas, ICMAT (CSIC-UAM-UC3M-UCM). Nicol\'as Cabrera, 13--15. 28049, Madrid, Spain}
\email{javier.aramayona@icmat.es}
\author[G. Domat]{George Domat}
\address{George Domat: Department of Mathematics, University of Michigan. 530 Church Street, Ann Arbor, MI 48109-1043, USA}
\email{domatg@umich.edu}
\author[C. J. Leininger]{Christopher J Leininger}
\address{Christopher J. Leininger: Department of Mathematics, University of Notre Dame, 255 Hurley Bldg
Notre Dame, IN 46556 USA}
\email{cleining@nd.edu}
\thanks{
JA was supported by grant PID2024--155800NB--C31, and acknowledges financial support from the Spanish Ministry of Science and Innovation, through the ``Severo Ochoa Programme for Centres of Excellence in R\&D (CEX2023-001347-S)”. GD  was partially supported by NSF grant DMS–-2303262. CJL was supported by NSF grants DMS--2305286 and DMS--2639932, and DMS--1928930 while in residence at the Simons Laufer Mathematical Sciences Institute in Berkeley, California, during the Spring 2026 Semester.
}

\date{\today}

\begin{document}

\begin{abstract}
We give conditions ensuring that an asymptotically rigid mapping class group, specifically a surface Houghton group $\mathcal H (S)$, has finite index in the asymptotically conformal modular group $\Mod_0(X)$, where $X$ is a hyperbolic structure on $S$. These include geometric conditions on the pieces of the underlying rigid structure, as well as the existence in $\SHo(S)$ of an end-periodic homeomorphism which is asymptotically conformal. As a consequence, if $S$ has $n\ge 3$ ends, then $\Mod_0(X)$ has type $F_{n-1}$ but not $FP_n$. We also establish analogous results for $L^p$ modular groups.
\end{abstract}
\maketitle

\addtocontents{toc}{\protect\setcounter{tocdepth}{0}}
\section{Introduction}

There are two groups naturally associated to a Riemann surface $X$: the {\em mapping class group} $\Map(X)$, the group of isotopy classes of homeomorphisms of $X$; and the {\em Teichm\"{u}ller modular group} $\Mod(X)$, consisting of those elements of $\Map(X)$ with a quasiconformal representative. When $X$ has finite type (that is, when its fundamental group is finitely generated), the two groups are well-known to coincide; however, when $X$ has infinite topological type, $\Mod(X)$ is an infinite-index subgroup of $\Map(X)$.

The main purpose of this paper is to prove that prescribing the ``behavior at infinity" of elements of $\Map(X)$ and $\Mod(X)$, in a suitable sense in each case — combinatorial for the former, analytic for the latter — often gives rise to commensurable subgroups. Before stating our results, we give an abridged introduction of the groups in question, postponing the formal definitions until \Cref{sec:background}. 

Let $S$ be a surface with finitely many ends, all of them non-planar. A {\em rigid structure} $\mathcal R$ on $S$ is, in essence, a decomposition of $S$ into {\em pieces} homeomorphic to a fixed compact surface $P$, together with a collection of {\em marking homeomorphisms} identifying each piece with $P$. The {\em surface Houghton group} $\mathcal H(S,\mathcal R)$ is the subgroup of $\Map(S)$ whose elements eventually send pieces to pieces {\em rigidly}, that is, compatibly with the markings outside some compact subsurface.

Now endow S with a hyperbolic structure $X$ — in particular, a Riemann surface structure — and let $\Mod_0(X) \le \Mod(X)$ denote the subgroup of asymptotically conformal classes, that is, those admitting a representative whose Beltrami coefficient vanishes at infinity.

\medskip

\noindent{\bf Asymptotically conformal vs asymptotically rigid.}
 Our first theorem asserts that, whenever the pieces of the rigid structure have nice geometry with respect to 
$X$, these two restrictions give rise to essentially the same group.  See \Cref{ssec:bddpieces} for the definition of {\em bounded pieces}.

\begin{customthm}{A}
    \label{thm:mainbddpieces}
    Let $S$ be a surface with finitely many ends, each of which is accumulated by genus. If $\mathcal{R}$ is a rigid structure on $S$, $X$ is a hyperbolic Riemann surface structure on $S$ so that the pieces of $\mathcal{R}$ are bounded, and $\mathcal{H}(S,\mathcal{R}) < \Mod_0(X)$, then $[\Mod_0(X): \mathcal{H}(S,\mathcal{R})] < \infty$. 
\end{customthm}

We note that, while $\mathcal{H}(S,\mathcal{R})$ is countable, $\Mod_0(X)$ is in general uncountable, which shows that some condition on the geometry of the pieces of the rigid structure is necessary in the above theorem.  

\medskip

We will also prove a {\em $p$-integrable} version of \Cref{thm:mainbddpieces}, for $p\ge 2$. Here, $\Mod^p(X)$ is the subgroup of $\Mod(X)$ whose elements have $p$-integrable Beltrami coefficient. We will show: 

\begin{customthm}{B}
    \label{cor:mainlp}
    Let $S$ be a surface with finitely many ends, each of which is accumulated by genus, $\mathcal{R}$ be a rigid structure on $S$, $X$ a hyperbolic Riemann surface structure on $S$, and $p \geq 2$. If the pieces of $\mathcal{R}$ are bounded in $X$ and $\mathcal{H}(S,\mathcal{R}) < \Mod^p(X)$, then $[\Mod^p(X):\mathcal{H}(S,\mathcal{R})]<\infty$. 
    
\end{customthm}

\medskip

\noindent{\bf Asymptotically conformal vs eventually isometric.} We also will show that a stronger condition on the geometry of $X$ yields the sharper result that asymptotically conformal homeomorphisms are in fact {\em eventual isometries}. We note that in this setting the containment $\mathcal{H}(S,\mathcal{R})<\Mod_0(X)$ is satisfied automatically, and can therefore be dropped from the hypotheses.  See \Cref{sec:compatiblestructures} for the definition of {\em compatible structure}.

\begin{customthm}{C}
    \label{thm:maintight}
    Let $S$ be a surface with finitely many ends, each of which is accumulated by genus. Let $\mathcal{R}$ be a rigid structure on $S$ and $X$ an $\mathcal{R}$-compatible hyperbolic Riemann surface structure on $S$. If $f \in \Mod_0(X)$, then there exists a $K< S$ so that $f$ is an isometry on $S \setminus K$. 
\end{customthm}

\medskip

\noindent{\bf Finiteness properties.} Combining our results with the previous work of \cite{AraBuxKimLei} we can compute the finiteness properties of $\Mod_0(X)$ and $\Mod^p(X)$ in the cases above. In particular, in the cases above these groups are \emph{finitely generated} and once the surface has at least three ends, they are also \emph{finitely presented}. A priori, one would expect $\Mod_0(X)$ and $\Mod^p(X)$ to be uncountable groups in general. Indeed, once either group contains a multi-twist about infinitely many disjoint curves, they are uncountable. 

In fact, Epstein \cite{epstein2000} posed the question of whether the (reduced) modular group can ever by countable for an infinite-type surface. Matsuzaki \cite{Matsuzaki2005} answered this in the affirmative by constructing an example where $\Mod(X)$ is countable for $X$ homeomorphic to the sphere minus a Cantor set. This was then extended to all infinite-type surfaces of finite genus in \cite{NinoHernandez2024}. 

\begin{customcor}{D}
    \label{cor:mainfinitenessprop}
    Let $S$ be the surface with a finite number  $n\ge 2$ of ends, each of which is accumulated by genus, $\mathcal{R}$ be a rigid structure on $S$, and $X$ a hyperbolic Riemann surface structure on $S$. Let $G$ be a finite index subgroup of either $\Mod_0(X)$ or $\Mod^p(X)$ for $p\geq 2$. If the pieces of $\mathcal{R}$ are bounded in $X$ and $\mathcal{H}(S,\mathcal{R}) < G$, then $G$ is of type $F_{n-1}$ but not $FP_n$.
\end{customcor}

\medskip

\noindent{\bf Mapping tori and end-periodic homeomorphisms} As noted in \cite{AraBuxKimLei}, every end periodic homeomorphism represents a mapping class that is conjugate into a surface Houghton group (c.f.~\cite[Corollary 2.9]{EndPeriodic1}).  End-periodic homeomorphisms have a Nielsen-Thurston like classification and arise naturally in the study of depth-one taut foliations and transverse pseudo-Anosov flows \cite{FenleyCT, Fenley-depth-one,CC-book,LandryMinskyTaylor2023}.

The results in \cite{EndPeriodic1,EndPeriodic2} are of particular relevance to this paper. In the finite-type surface setting, Brock \cite{Brock-convex-vol,Brock-mappingtorus-vol} first showed that given a pseudo-Anosov mapping class, $f \in \Map(S)$, the hyperbolic volume of the corresponding mapping torus, $M_f$, is coarsely comparable to the translation length of $f$ acting on either $\T(S)$ equipped with the Weil-Petersson metric or the pants graph of $S$, these spaces being quasi-isometric; see also Agol \cite{Ago03}. Now, when $S$ is an infinite-type surface with finitely many ends, all non-planar, the authors in \cite{EndPeriodic1,EndPeriodic2} prove that given a strongly irreducible, end-periodic mapping class $f \in \Map(S)$, one obtains a similar relationship between the hyperbolic volume of a natural compactification of the mapping torus, $\overline{M_f}$, and the translation length of $f$ on a suitable pants graph for $S$. This begs the following question.

\begin{question}
    Let $S$ be a surface with finitely many ends, all non-planar, and $f \in \Map(S)$ a strongly irreducible, end-periodic mapping class. Can one obtain upper and lower bounds on the hyperbolic volume of a natural compactification of the mapping torus, $\overline{M_f}$, in terms of the Weil-Petersson translation length of $f$ acting on some version of Teichm\"{u}ller space corresponding to $S$?
\end{question}

Since end-periodic homeomorphisms are all conjugate into some surface Houghton group $\SHo(S,\cR)$, we propose interpreting this question in the language of these groups, and the various Teichm\"uller spaces considered in this paper.
In particular, the square--integrable Teichm\"{u}ller space can be equipped with a form of Weil-Petersson metric \cite{Yanagishita2017}. So, \Cref{cor:mainlp} gives a natural setting in which to study the Weil-Petersson translation length of an end-periodic mapping class. 

One potential drawback to applying \Cref{cor:mainlp} is that one must verify that $\SHo(S,\cR)$ acts. However, we show that simply verifying that a single end-periodic mapping class acts is sufficient for the entire surface Houghton group to act, at least up to finite index.

\begin{customthm}{G}\label{thm:mainep}
 Let $S$ be the surface with a finite number  $n\ge 2$ of ends, each of which is accumulated by genus, $\mathcal{R}$ be a rigid structure on $S$, and $X$ a hyperbolic Riemann surface structure on $S$ so that the pieces of $\cR$ are bounded. If $f \in \SHo(S,\cR)$ is an end periodic homeomorphism with either $f \in \Mod_0(X)$ (or $f \in \Mod^p(X)$ for $p\geq 2$),  then $\Gamma =\Mod_0(X) \cap \SHo(S,\cR)$ (resp. $\Mod^p(X) \cap \SHo(S,\cR)$) is finite-index in both $\Mod_0(X)$ (resp. $\Mod^p(X)$) and $\SHo(S,\cR)$. In particular, $\Mod_0(X)$ (resp. $\Mod^p(X)$) has type $F_{n-1}$ and not type $FP_n$.    
\end{customthm}

\noindent{\bf Further questions.} Here we collect a number of additional questions related to $\T_0(X)$, $\Mod_0(X)$, and surface Houghton groups. The first is to find a sharp version of \Cref{thm:mainbddpieces}. 

\begin{question}
    Given a rigid structure $\cR$ on $S$, what are necessary and sufficient conditions on a hyperbolic structure $X$ on $S$ to ensure that $\SHo(S,\cR)$ and $\Mod_0(X)$ are commensurable? In particular, if $\SHo(S,\cR) < \Mod_0(X)$, is $X$ having bounded geometry sufficient?  Or if $X$ has a bounded length pants decomposition?     
\end{question}

We note that some condition on the geometry of $X$ is necessary. By carefully controlling the rate at which certain curves shrink as one moves out an end, one can produce examples where $\SHo(S,\cR)$ acts, but $\Mod_0(X)$ is uncountable. See \Cref{S:unbounded pieces} and \Cref{Q:unbounded pieces} for a concrete example for which we know that $\SHo(S,\cR)$ acts on $\T_0(X)$, but do not know if it is commensurable with $\Mod_0(X)$. 

One can also flip this around and ask the following related question.  Again, the examples in \Cref{S:unbounded pieces} suggest that the answer to this question is likely to be subtle.

\begin{question}
    Suppose that for some rigid structure, $\cR$, and some hyperbolic Riemann surface structure, $X$, $\SHo(S,\cR) < \Mod_0(X)$. Does this force $X$ to have certain geometric features? What if $[\Mod_0(X):\SHo(S,\cR)] < \infty$?
\end{question}

Finally, in this paper we have restricted attention to surfaces with finitely many ends and their associated surface Houghton groups. However, one can consider more exotic surfaces and their corresponding asymptotically rigid mapping class groups. For example, in \cite{ABFPW2024}, the authors consider asymptotically rigid mapping class groups of \emph{Cantor surfaces}, e.g. surfaces with a Cantor set of ends, all non-planar.  Examples of hyperbolic structures $X$ for which the associated asymptotically rigid mapping class groups act were studied by \cite{dFGH2004}.  These examples fail to have finite index; however a variation of asymptotically rigid mapping class groups (for instance, suitable subgroups of the {\em block mapping class groups} with finite local groups from \cite{AACKW2025}) may in fact be finite index in these examples.

\begin{question}
    Let $S$ be a surface with a Cantor set of ends, all non-planar. Does there exist a variant of \Cref{thm:mainbddpieces} for a version of asymptotically rigid mapping class groups of $S$?
\end{question}

\subsection{Outline}

    In \Cref{sec:background} we recall relevant background and fix notation on Teichm\"{u}ller spaces, modular groups, and surface Houghton groups. In \Cref{sec:compatiblestructures} we introduce compatible hyperbolic structures and give the proof of \Cref{thm:maintight}. As a consequence we obtain \Cref{thm:mainbddpieces} in the case when $X$ is a \emph{compatible} hyperbolic surface. In this way, \Cref{thm:maintight} can be seen as something of a ``warm-up'' for the proofs to come. 

    Next, in \Cref{sec:genhypbp} we introduce the notion of $X$ having \emph{bounded pieces} as a weakening of $X$ being a compatible hyperbolic structure. Roughly speaking, $X$ having a compatible structure guarantees that the pieces that make up the rigid structure on $X$ are actually isometric. Meanwhile, bounded pieces weakens this significantly to only require that each piece has bounded diameter in $X$. This section also serves as the main technical machinery of the paper. We prove that while individual pieces in $X$ may have poorly behaved \emph{intrinsic} geometry, they can be enveloped in slightly larger subsurfaces, we call \emph{envelopes}, whose intrinsic geometry and topology is controlled. This is done in \Cref{ssec:envelopes}. Here we also include some additional results that explore the geometry obtained by having bounded pieces. Some of these results are not strictly necessary for the proof of \Cref{thm:mainbddpieces}, but they may prove useful for other applications.
    We then prove \Cref{lem:identity propogates} which allows us to check whether a mapping class is the identity on a neighborhood of an end by reducing to checking that it is the identity on a single piece.

    With the tools of the previous section to hand, we obtain a fairly short proof of \Cref{thm:mainbddpieces} in \Cref{sec:bpandsho}. This section also includes the proof of \Cref{thm:mainep}. Next, in \Cref{sec:lp} and following \cite{Yanagishita2014}, we introduce $L^p$ quasiconformal deformation spaces and their associated $L^p$ modular groups. Making use of some results in \cite{Yanagishita2014} we can upgrade the results of the previous sections in order to obtain \Cref{cor:mainlp}. Finally in \Cref{sec:examples} we give a menagerie of examples of hyperbolic surfaces in order to elucidate the different conditions we require of hyperbolic structures in \Cref{sec:compatiblestructures} and \Cref{sec:genhypbp}. We also give an example in \Cref{S:unbounded pieces} of a hyperbolic structure for which $\SHo(S,\cR)<\Mod_0(X)$ and yet we do not know if the conclusion of \Cref{thm:mainbddpieces} holds. 

\subsection*{Acknowledgements}

GD thanks Evangelos Nikitopoulos for patiently answering his complex analysis questions. GD and CL thank Ken Bromberg for helpful conversations on Teichm\"{u}ller geometry. 



\section{Background}\label{sec:background}

In this section, we provide a brief overview of the necessary background. Throughout this article, by a {\em (topological) surface} we mean a connected, second-countable, orientable 2-manifold, possibly with non-empty boundary, which consists of a finite number of circles. We recall (see \cite{AramayonaVlamis}, for instance) that the homeomorphism type of a surface is determined by its genus, number of boundary components, and the topology of its \emph{space of ends}. We will always assume that $S$ is a surface with a finite number $n$ of ends, each of which is accumulated by genus. 

\subsection{Teichm\"uller space}\label{ssec:teich}

There are many resources on the basics of Teichm\"{u}ller space, but we direct the interested reader to the books of Gardiner-Lakic \cite{GardLak} and Nag \cite{Nag1988} as they, in general, do not assume that surfaces are of finite type. Suppose $X = \hyp/\Gamma$ is a hyperbolic Riemann surface structure on a topological surface $S$, where $\Gamma$ is a Fuchsian group acting on the upper half-plane, $\hyp$.  More precisely, in this we implicitly assume that there is specific homeomorphism $S \cong X$ which we use to transfer the structure to $S$.  We may also simply refer to the hyperbolic Riemann surface $X$ when the specified homeomorphism with $S$ is unimportant.

We will consider $X$ as both a Riemann surface structure and a metric surface with its canonical hyperbolic metric; anytime we make reference to a metric on $X$ we are referring to this hyperbolic metric. Let $\CL^{\infty}(X)$ be the space of measurable $(-1,1)$-differentials on $X$ with finite $\CL^{\infty}$-norm. That is,
\begin{align*}
    ||\mu||_{\infty} = \esup_{z \in X} |\mu(z)| < \infty.
\end{align*}
We could equally well work with $\CL^\infty(\hyp,\Gamma)$, the space of $\Gamma$--invariant $(-1,1)$-differentials on $\hyp$; indeed, there is a canonical isomorphism $\CL^\infty(X) \to \CL^\infty(\hyp,\Gamma)$, by pull-back, and we often pass back and forth between these (and all other spaces of differentials) in this way.
 
We use $\Bel(X)  \cong \Bel(\hyp,\Gamma)$ to denote the space of {\em Beltrami coefficients} on $X$ (or measurable Beltrami coefficients for $\Gamma$); that is, the open unit ball in $\CL^{\infty}(X)$.  The Measurable Riemann Mapping Theorem implies that for every $\mu \in \Bel(\hyp,\Gamma)$, there is a quasiconformal self map $\widetilde f^\mu$ of $\hyp$ fixing $0$, $1$, and $\infty$, with Beltrami coefficient $\mu$ that conjugates $\Gamma$ to another Fuchsian group, $\Gamma^\mu$.  The map $\widetilde f^\mu$ descends to a quasiconformal map
\[ f^\mu \colon X \to X^\mu = \hyp/\Gamma^\mu,\]
with Beltrami coefficient $\mu \in \Bel(X)$.  In general, given a quasiconformal map $f \colon X \to X'$, we write $\mu_f$ for its Beltrami coefficient.

We assume $X$ is of the first kind (that is, the limit set of $\Gamma$ is the entire circle at infinity), and say two Beltrami differentials $\mu,\nu \in \Bel(X)$ are \emph{Teichm\"{u}ller equivalent} if the associated quasiconformal maps of $\hyp$ have extensions that satisfy $\widetilde f^{\mu}\vert_{\R}=\widetilde f^{\nu}\vert_{\R}$. The \emph{Teichm\"{u}ller space} (or \emph{quasi-conformal deformation space}) of $X$, denoted $\T(X)$, is the space of Teichm\"{u}ller equivalence classes of Beltrami coefficients.  Alternatively, we can consider points in $\T(X)$ as equivalence classes of quasi-conformal homeomorphism $f \colon X \to X'$ where $f \colon X \to X'$ is equivalent to $h \colon X \to X''$ if $fh^{-1} \colon X'' \to X'$ is homotopic to a conformal map.  The association $[\mu] \mapsto [f^\mu \colon X \to X^\mu]$ provides the required identification between these two perspectives.  See, for example, \cite[Section 2.4]{GardLak} for more details.  We write $[\mu] = [f^\mu \colon X \to X^\mu] = [f^\mu,X^\mu]$ to represent the same point, in different contexts.

If $X$ is a hyperbolic Riemann surface and $\gamma$ a closed curve on $X$ we let $\ell_X(\gamma)$ denote the length of the geodesic representative of the homotopy class of $\gamma$ in the hyperbolic metric on $X$. This is an invariant of the homotopy class of $\gamma$. 

\begin{definition} If $f \colon X \to X'$ is a homeomorphism and $L \geq 1$, we say that $f$ is {\em $L$--tight}, if for all closed geodesics $\gamma$ in $X$, we have
\[ \frac{1}{L} \leq \frac{\ell_{X'}(f(\gamma))}{\ell_X(\gamma)}\leq L.
\]
We write $\ell_Y(f(\gamma)) \stackrel{L}{\asymp} \ell_X(\gamma)$ to mean that both inequalities above hold.
\end{definition}

We will make ample use of the following lemma due to Wolpert \cite[Lemma 3.1]{Wolpert1979}, which allows one to relate quasiconformality and hyperbolic lengths of curves. 
\begin{lemma}\cite[Lemma 3.1]{Wolpert1979}\label{lem:wolpert}
    Let $X$ and $X'$ be hyperbolic Riemann surfaces. If $f:X \rightarrow X'$ is $K$-quasiconformal, then it is $K$--tight.
\end{lemma}

\begin{remark}
We note that while the property that $f$ is $K$-quasiconformal is a property of the map $f$, being $L$-tight is a property of the isotopy class of $f$. Thus, given an isotopy class of maps $[f \colon X \to X']$, we may refer to this isotopy class as being $L$--tight, and a representative homeomorphism $f \colon X \to X'$ being $K$--quasiconformal.
\end{remark}

\subsection{Modular groups}\label{ssec:modgrps}

Given a hyperbolic Riemann surface $X$, recall that $\Mod(X)$ is the Teichm\"uller modular group of $\T(X)$, consisting of mapping classes that are represented by quasiconformal homeomorphisms with respect to the structure $X$. 
A result of Markovic \cite{Markovic2003} implies that $\Mod(X)$ is exactly the biholomorphic automorphism group of $\T(X)$, except in some finite type cases that will not be relevant for the current paper.  Viewed as mapping classes of quasi-conformal homeomorphisms, $\Mod(X)$ acts on $\T(X)$ by
\[ g \cdot [f \colon X \to X'] = [f \circ g^{-1} \colon X \to X'], \]
which is well defined since $f \circ g^{-1}$ is quasi-conformal.

\subsection{Asymptotically conformal Teichm\"uller space}\label{ssec:acteich}

In \cite{EGL2000,EGL2004}, Earle-Gardiner-Lakic  defined a natural subspace of $\T(X)$ consisting of asymptotically conformal deformations of $X$ (see also \cite[Chapter 14]{GardLak}). 

\begin{definition}
    A quasiconformal map $f: X \rightarrow X'$ is \textbf{asymptotically conformal} if for every $\epsilon >0$, there exists a compact set $C\subset X$ such that $f\vert_{X \setminus C}$ is $(1+\epsilon)$-quasiconformal.
\end{definition}

The \emph{asymptotically conformal Teichm\"{u}ller space} (or \emph{asymptotically conformal deformation space}) of $X$, denoted $\T_{0}(X)$, is the closed subspace of $\T(X)$ consisting of classes $[f,X']$ that contain an asymptotically conformal representative. The fact that this is a \emph{closed} subspace is due to Gardiner-Sullivan \cite[Proposition 3.3]{GS1992}. In fact, $\T_{0}(X)$ is a closed complex submanifold of $\T(X)$ \cite{EGL2000} and is connected and complete with respect to the Teichm\"{u}ller metric \cite{EGL2004}. Earle-Gardiner-Lakic also prove similar results about the \emph{asymptotic Teichm\"{u}ller space} of $X$, which is a quotient of $\T(X)$ for which the fiber over the identity map is $\T_{0}(X)$.  We will not consider this latter space here.

Note that a map $f:X \rightarrow X'$ is asymptotically conformal if and only if its Beltrami coefficient, $\mu_{f}$, vanishes at infinity. That is, for every $\epsilon>0$, there exists a compact set $C \subset X$ such that $|\mu_{f}(z)| < \epsilon$ for all $z \in X \setminus C$. We will write $\Bel_{0}(X)$ for the subset consisting of Beltrami coefficients on $X$ that vanish at infinity. Thus a map $f:X \rightarrow X'$ is asymptotically conformal if and only if $\mu_{f} \in \Bel_{0}(X)$.

The composition of a $K$--quasiconformal with a $K'$--quasi-conformal map is $KK'$--quasiconformal.  The following is a straightforward consequence which we record for future use. 

\begin{lemma} \label{L:ACgroup}
\begin{enumerate}
    \item If $f: X \to X'$ is asymptotically conformal, then so is its inverse $f^{-1}: X'\to X$. 
    \item If $f_1:X_1 \to X_2$ and $f_2:X_2\to X_3$ are asymptotically conformal, then so is the composition $f_2\circ f_1: X_1 \to X_3$. \qed
\end{enumerate}
\end{lemma}
Asymptotically conformal homeomorphisms can be adjusted to better describe the behavior of hyperbolic metrics by the following theorem of Earle--Markovic--Saric \cite{EMS2002}.  We say that a biLipschitz homeomorphism $f \colon X \to X'$ between hyperbolic Riemann surfaces is an {\em asymptotic isometry} if for every $K > 1$ there exists a compact subset $C \subset X$ so that $f|_{X \setminus C}$ is $K$--biLipschitz.  An asymptotically isometric map is in particular asymptotically quasiconformal.

\begin{theorem}\cite[Theorem 3]{EMS2002}\label{thm:EMS}
   Let $X$ be a hyperbolic Riemann surface. Every point of $\T_0(X)$ can be represented by an asymptotic isometry. 
\end{theorem}

In fact, the representative in the theorem is (the descent of) the Douady-Earle extension of the boundary values for an asymptotically conformal homeomorphism (\cite{DouadyEarle}). We will also use the following, which can be viewed as a strengthening of Wolpert's lemma for asymptotically conformal deformations.  This is an immediate consequence of Theorem~\ref{thm:EMS}, though it can also be proved by a similar argument to Wolpert's original one from \cite{Wolpert1979} together with the monotonicity of modulus for annuli.

\begin{corollary} \label{cor:nearly isometric}
    Let $X$ be a hyperbolic Riemann surface.  For any asymptotically conformal map $f:X \rightarrow X'$ and $L > 1$, there exists a compact $C \subset X$ such that, for all closed geodesics $\gamma \subset X \setminus C$,
    \[ \ell_{X'}(f(\gamma)) \stackrel{L}{\asymp} \ell_X(\gamma).\]
\end{corollary}

\begin{definition}
    For any essential subsurface $\Sigma \subset X$ with geodesic boundary and $L>1$, we say that $f:X \rightarrow X'$ is \emph{$L$-tight on $\Sigma$} if 
    \[ \ell_{X'}(f(\gamma)) \stackrel{L}{\asymp} \ell_X(\gamma)\]
    for all $\gamma$ closed geodesics contained in $\Sigma$. 
\end{definition}
Thus, \Cref{cor:nearly isometric} says that given an asymptotically conformal map, $f$, and any $L>1$, there exists some compact $C \subset X$, so that, $f$ is $L$-tight on $X \setminus C$. 

\subsection{Asymptotically conformal modular groups}\label{ssec:acmodgrps}

We write
\[ \Mod_{0}(X) < \Mod(X)\]
to denote the subgroup consisting of mapping classes represented by asymptotically conformal homeomorphisms with respect to the structure $X$. We will refer to this group as the \emph{asymptotically conformal modular group}. We use $\PMod_{0}(X)$ to refer to the subgroup of $\Mod_{0}(X)$ consisting of mapping classes that fix the ends of $S$ pointwise. Note that if $S$ has finitely many ends, then $\PMod_{0}(X)$ is a finite index subgroup of $\Mod_{0}(X)$.

\subsection{Surface Houghton Groups}\label{ssec:shogrps}
Here we will define rigid structures on surfaces with finitely many ends, all of them non-planar. In this specificity, these rigid structures give rise to \emph{surface Houghton groups} as defined in \cite{AraBuxKimLei,AraDomLei}. We remark that these definitions appear in various forms in \cite{ABFPW2024,AraBuxKimLei,AraDomLei}, which are in turn inspired by the homonymous notion due to Funar--Kapoudjian \cite{Funar-Kapoudjian, FunKap}. 

Fix a compact, connected, orientable surface $O$, which we call the  {\em core} surface, with at least one boundary component. Let $P$, referred to as the {\em model piece}, be a compact, connected surface with two boundary components $\partial_{+}P$ and $\partial_-P$, and referred to as the \emph{top} and \emph{bottom} boundaries, respectively. We fix a homeomorphism $\lambda:\partial_{-} P \rightarrow \partial_{+} P$. 

Let $S$ be a connected orientable surface with finitely many ends, each of which is non-planar. A \emph{rigid structure}, denoted $\cR$, on $S$ is a decomposition
\begin{align*}
    S = O \cup \bigcup_{j\in J} Y_{j}
\end{align*}
 of $S$ into compact subsurfaces, where $J$ is a countable index set, together with a set of fixed \emph{marking homeomorphisms} $\phi_{j}:Y_{j} \rightarrow P$, for each $j\in J$. We use these homeomorphisms to label the boundary components of each $Y_{j}$ as $\partial_{\pm}Y_{j} = \phi_{j}^{-1}(\partial_\pm P)$. We additionally require this decomposition to satisfy the following:
\begin{itemize}
    \item $\inte(Y_i) \cap \inte(Y_j)= \emptyset$ for $i\ne j$.
    \item For all $j\in J$, either $Y_{j} \cap O = \emptyset$ or $Y_{j} \cap O = \partial_{+}Y_{j}$ is a component of $\partial O$. 
    \item For all $i \neq j \in J$, either $Y_{i} \cap Y_{j} = \emptyset$ or $Y_{i} \cap Y_{j} = \partial_{+}Y_{j} = \partial_{-}Y_i$, and
    \begin{align*}
        \phi_{j}^{-1}\vert_{\partial_{+}P} \circ \lambda \circ \phi_{i}\vert_{\partial_{-}Y_{i}} = \id\vert_{\partial_{-}Y_{i}},
    \end{align*}
    after possibly swapping the roles of $i$ and $j$.
\end{itemize}

We refer to each $Y_j$ as a \emph{piece}. A \emph{suited subsurface} of $S$ is a connected union of $O$ and finitely many pieces. 

We say that a mapping class $h\in \Map(S)$ is \emph{asymptotically rigid} if it has a representative (abusing notation) $h: S \to S$ for which there there exists a suited surface $Z$, called a \emph{defining surface}, satisfying: 
\begin{itemize}
    \item $h(Z)$ is a suited subsurface of $S$; 
    \item $f$ is {\em rigid away from} $Z$, i.e. for every piece $Y_j$ contained in $S \setminus Z$, we have that $f(Y_i)$ is a piece $Y_j$, and $h_{|Y_i}\equiv \phi_j^{-1} \circ \phi_i$. 
\end{itemize}

The \emph{surface Houghton group} of $S$ with respect to the rigid structure $\cR$, denoted $\SHo(S,\cR)$, is the subgroup of $\Map(S)$ consisting of asymptotically rigid mapping classes. We only ever consider a single fixed rigid structure and surface Houghton group at a time and so we often omit $\cR$ and simply write $\SHo(S)$. We note that these groups are examples of a more general class of \emph{asymptotically rigid mapping class groups}, e.g. see \cite{FunKapSer2012,GenLonUre2022,GenLonUre2025,ABFPW2024}.

Surface Houghton groups were studied in \cite{AraBuxKimLei,AraDomLei}, mainly from the point of view of their finiteness properties. In particular, if $S$ is a surface with exactly $n$ ends, all accumulated by genus, then $\SHo(S)$ (independent of the rigid structure) is of type $F_{n-1}$ but not of type $FP_n$. The actual computation of these finiteness properties is from \cite[Theorem 1.1]{AraBuxKimLei} and the independence of rigid structure is \cite[Theorem 1.1]{AraDomLei}.

Next we introduce some useful terminology and constructions related to rigid structures and these groups. Let $S$ be the surface with exactly $n\ge 1$ ends, all accummulated by genus. A (closed) neighborhood of an end $\E$ that is the connected union of pieces will be called a {\em combinatorial neighborhood} of the end.  We will denote a combinatorial neighborhood by
\[ U = Y_1 \cup Y_2 \cup \ldots, \]
where $Y_i,Y_{i+1}$ are {\em adjacent} (i.e.~intersecting) pieces in the union.  In particular, $Y_1$ is not adjacent to any other piece and we call this the {\em boundary piece} of $U$, while every other piece is in the interior of $U$.
We sometimes say that the pieces $Y_1,Y_2,\ldots$ are {\em enumerated linearly} in $U$.

In \cite{aramayona2020first}, Aramayona-Patel-Vlamis construct a {\em flux homomorphism}
\[ \Phi \colon \PMap(S) \to {\mathbb Z}^n,\]
whose $k^{th}$ coordinate measures the {\em genus shift} out the $k^{th}$ end, $\E_k$. Here $\PMap(S)$ is the kernel of the action of $\Map(S)$ on the space of ends, i.e. the \emph{pure mapping class group}; observe that $\PMap(S)$ has finite index if $S$ has finitely many ends. We also let $\PSHo(S) = \SHo(S) \cap \PMap(S)$ be the \emph{pure surface Houghton group}. 
For any element $f \in \PSHo(S)$, every end has a combinatorial neighborhood so that $f$ shifts every piece away or towards the core by some integer $j_k(f)$ amount, depending only on the end $\E_k$ and the mapping class $f$ (with positive integers shifting away from the core, and negative integers shifting towards it).  If a piece defining $S$ has genus $g$, then 
\[ \Phi(f) = (gj_1(f),\ldots,gj_n(f)). \]

From the discussion in \cite[Section 6]{AraBuxKimLei}, using results in \cite{aramayona2020first}, one can obtain a fairly simple description of a generating set for $\SHo(S)$. First one takes a finite collection, $\Omega$, of elements in $\SHo(S)$ whose action on the end space of $S$ is transitive. Next, given any end of $S$, pick a combinatorial neighborhood $U=Y_1 \cup Y_2 \cup \cdots$ and take an element $F$ that acts on $U$ by a rigid map sending $Y_k$ to $Y_{k+1}$ for all $k \geq 1$. The group $\SHo(S)$ is generated by $\Omega$, $F$, and the set of compactly supported mapping classes. Note that this is an \emph{infinite} generating set, but this will be sufficient for our needs. If one wants to write down a finite generating set, whose existence is guaranteed by \cite{AraBuxKimLei}, one needs to use the fact that Dehn twists generate the compactly supported mapping classes.

\section{Compatible hyperbolic structures}\label{sec:compatiblestructures}

Let $\cR$ be a rigid structure on a surface $S$ and $X$ a hyperbolic Riemann surface structure on $S$ (recall that this means there is, in particular, an implicit homeomorphism from $S$ to $X$ we use to identify the two surfaces). 

We say that a combinatorial neighborhood $U = Y_1 \cup Y_2 \cup \cdots \subset X$ of an end $\E$ is \emph{$\cR$-compatible} if there is a fixed hyperbolic structure on the model piece $P$ such that $\lambda:\partial_- P \to \partial_+P$ is an isometry, and so that each marking homeomorphism $\phi_k \colon Y_k \to P$ is an isometry (where $Y_k$ is given the restriction of its hyperbolic metric from $X$).
Note that the intersection of any finite set of $\cR$-compatible neighborhoods (of the same end) is itself a $\cR$-compatible neighborhood.

We say that $X$ is \emph{$\cR$-compatible} if each end of $X$ has an $\cR$-compatible neighborhood with respect to the same fixed hyperbolic structure on $P$. When $\cR$ is understood we omit it and will refer to $X$ simply as a \emph{compatible hyperbolic structure} and such a $U$ as a \emph{compatible neighborhood} of the end $\E$. Note that $X$ being compatible is equivalent to saying that $\phi_j \colon Y_j \to P$ is an isometry for all but finitely many $j \in J$ (where the hyperbolic metric on $P$ and $\lambda$ have the properties above).

A homeomorphism $f: X \to X$ will be called {\em eventually isometric for $X$} (or just {\em eventually isometric}, if $X$ is understood), if it is quasi-conformal on $X$, and is an {\em isometry} outside some compact subset.  We also call an element of $\Mod(X)$ eventually isometric if it has an eventually isometric representative.

\begin{remark}
    We note that if $f \colon X \to X$ is an eventual isometry, then it is an asymptotic isometry, but not vice-versa.
\end{remark}

For the remainder of this section, we fix a rigid structure $\cR$ on a surface $S$ with finitely many ends, each of which is accumulated by genus, and omit any reference to  $\cR$ when understood. 

\begin{lemma} \label{lem:tight basics}
Suppose $X$ is a compatible hyperbolic structure.  Then $\SHo(S)<\Mod(X)$ and every element is eventually isometric. In particular, $\SHo(S)< \Mod_{0}(X)$.
\end{lemma}
\begin{proof}
Fix $f \in \SHo(S)$, there exists a (compact) suited subsurface $Z \subset X$ so that $f$ is rigid outside $Z$.  By enlarging $Z$ if necessary, we can further assume that the closure of the complement of $Z$ is a union of compatible neighborhoods of the ends.  Let $Z' \subset X$ be an even larger suited subsurface with $Z \subset Z' \cap f(Z')$.  Then for any piece $Y_i \subset \overline{X \setminus Z'} \subset \overline{X \setminus Z}$, $f(Y_i) = Y_j$ is a piece and $Y_j \subset \overline{X \setminus f(Z')} \subset \overline{X \setminus Z}$.  In particular, because $\overline{X \setminus Z}$ is a union of compatible neighborhoods and $f$ is rigid outside $Z$, we see that
\[ f|_{Y_i} = \phi_j^{-1} \circ \phi_i \colon Y_i \to Y_j.\]
Both $\phi_i$ and $\phi_j$ are isometries, and hence so is $f|_{Y_i}$.  Since $Y_i$ was an arbitrary piece outside $Z'$, $f$ is an isometry on $\overline{X \setminus Z'}$, as required.

For the second claim, we simply note that we may apply an isotopy to $f$ rel $\partial Z'$, supported on $Z'$ so that it is a diffeomorphism, hence quasiconformal on $Z'$.  Since $f$ is conformal outside $Z'$, it is asymptotically quasiconformal.  Thus $\SHo(S) < \Mod_0(X)$.  

\end{proof}

\begin{lemma} \label{lem:rigid canonical}
If $X$ and $X'$ are any two compatible hyperbolic Riemann surface structures on $S$, then $\T(X) = \T(X')$; More precisely, composing the implicit homeomorphisms $S \cong X$ and $S \cong X'$ determines a homeomorphism $X \cong X'$ that is quasiconformal.
\end{lemma}
\begin{proof}
The compatibility of $X$ and $X'$ gives two hyperbolic structures on the model piece, $P$. After adjusting by an isotopy, we may assume that the identity on $P$ is $K$--biLipschitz with respect to these two structures, for some $K > 1$. Now by shrinking compatible neighborhoods of ends if necessary, we may assume that $X \cong X'$ sends compatible neighborhoods to compatible neighborhoods.  Moreover, this homeomorphism sends pieces to pieces by the identity (via the implicit homeomorphisms with $S$), and thus it follows that on the chosen compatible neighborhoods, the homeomorphism is $K$--biLipschitz.  Since the complement of these neighborhoods has compact closure, we can adjust it by an isotopy to be a diffeomorphism.  In particular, $X \cong X'$ is a biLipschitz homeomorphism, and in particular it is quasiconformal.

\end{proof}

Our goal is to show that if $X$ is a compatible hyperbolic structure, then $\SHo(S)$ is a finite-index subgroup of $\Mod_0 (X)$. In the presence of such a constrained hyperbolic structure, we actually show something stronger. 

\begin{theorem} \label{thm:rigid to tight}
If $X$ is a compatible hyperbolic structure and $f \in \Mod_{0}(X)$, then $f$ is eventually isometric for $X$.
\end{theorem}

We first need a preliminary lemma. Recall that the {\em (unmarked) length spectrum} of a hyperbolic surface $X$ is the set of the lengths of all closed geodesics on $X$. 

\begin{lemma} \label{lem:discrete lengths}
If $X$ is a compatible hyperbolic structure, then the length spectrum (without multiplicities) is a discrete, closed subset of $\mathbb R$.
\end{lemma}
\begin{proof}
Suppose $U_1,\ldots,U_n$ are, respectively, compatible neighborhoods of the ends $\E_1,\ldots,\E_n$ of $X$ and let $Z = \overline{X\setminus (U_1 \cup \cdots \cup U_n)}$ be the complementary compact subsurface.  Given $R > 0$ and $j=1,\ldots,n$ let $U_j(R) \subset U_j$ be a compatible neighborhood such that the distance from $U_j(R)$ to $\partial U_j$ is at least $R$, and denote by $Z_R = \overline{X \setminus (U_1(R) \cup \cdots \cup U_2(R))}$ the complementary compact subsurface.

Now suppose $\gamma$ is a geodesic in $X$ of length at most $R$.  Then either $\gamma$ is contained in $Z_R$ or is contained in some $U_j$.  Since $U_j$ is obtained by isometrically gluing copies of the fixed hyperbolic structure on $P$ using the isometry $\lambda$, it follows that if $U_j = Y_1 \cup Y_2 \cup \ldots$, then for any $k,m > 0$, there is an isometry $Y_1\cup \ldots \cup Y_k \to Y_{m+1} \cup \ldots \cup Y_{m+k}$ in the restriction of the hyperbolic structures from $X$.  Now any length $R$ curve is contained in $k$ consecutive surfaces $Y_{m+1} \cup \ldots \cup Y_{m+k}$, where $k$ depends only on $R$ (to see this, note that the boundary curves of each $Y_k$ admits a uniform collar neighborhood, and to pass from a piece to an adjacent one, it must cross this collar).  In particular, $\ell_X(\gamma) = \ell_X(\gamma')$ for some $\gamma$ that is contained in $Y_1 \cup \ldots \cup Y_k$.  Since this surface is compact, there are only finitely many {\em lengths of} curves outside $Z_R$ of length at most $R$.

\end{proof}

We are now ready to prove \Cref{thm:rigid to tight}: 

\begin{proof}[Proof of \Cref{thm:rigid to tight}]
For each end $\mathcal E$ of $X$, choose a compatible neighborhood $U_{\E} = Y_1^\E \cup Y_2^\E \cup \ldots$.  Let $f_\E \in \SHo(S)$ be an element mapping $U_\E$ isometrically into itself, sending $Y_k^\E$ isometrically to $Y_{k+1}^\E$, for all $k \geq 1$.

Let $Z_k^\E = Y_k^\E \cup Y_{k+1}^\E$ be the connected union of two consecutive pieces in $U_\E$.  There are finitely many closed curves $\gamma_1^\E,\ldots,\gamma_d^\E$ in $Z_1^\E$ whose lengths determine the hyperbolic structure on $Z_1^\E$.  Then for all $k \geq 1$, the collection $f_\E^k(\gamma_1^\E),\ldots,f_\E^k(\gamma_d^\E)$ is a set of curves in $Z_k^\E$ whose lengths determine the hyperbolic structure on $Z_k^\E$.  Let $R$ be the maximum value of the finitely many {\em lengths} of these curves $\{\ell_X(f_\E^k(\gamma_i^\E) \mid \E,i,k\}$.
By \Cref{lem:discrete lengths}, the set of lengths of closed geodesics $X$ is a closed discrete set.  Suppose $R$ is the $m^{th}$ shortest length of a geodesic in $X$, and write
\[ r_0 < r_1 < \cdots < r_{m-1} = R < r_m.\]
(Thus, $r_m$ is the next shortest length of a geodesic in $X$ after $R$.)  
Now, let $K $ be any real number with
\[1 < K < \inf\left\{ \tfrac{r_j}{r_{j-1}} \mid j = 1,\ldots,m \right\}.\]

Suppose $f \in \Mod_{0}(X)$ is any element.  By \Cref{cor:nearly isometric} there is a compact set $C \subset X$ so that all closed geodesics $\gamma$ in $X \setminus C$ have
\[ \ell_X(\gamma) \stackrel{K}{\asymp} \ell_X(f(\gamma)).\]
In particular, if $\gamma$ is a geodesic in $X\setminus C$ and $\ell_X(\gamma) \leq R$, then $\ell_X(f \cdot \gamma) = \ell_X(\gamma)$.  Without loss of generality, we may assume
\[ \overline{X \setminus C} = U_{\E_1} \cup \ldots \cup U_{\E_n}.\]

It follows that for every end $\E$ and $Z_k^\E \subset U_\E$,
\[ \ell_X(f \cdot f_\E^k(\gamma_i^\E)) = \ell_X(f_\E^k(\gamma_k^\E)).\]
Since the lengths of $f_\E^k(\gamma_1^\E),\ldots,f_\E^k(\gamma_d^\E)$ determine the hyperbolic structure on $Z_k^\E$, the same is true of their $f$--images in $f(Z_k^\E)$. Thus $f$ can be isotoped so that it restricts to an isometry of $Z_k^\E$ for every end $\E$ and every $k \geq 1$.  Since the $Z_k^\E$ cover all of $U_\E$, $f$ is an isometry on each neighborhood $U_\E$, and hence $f$ is eventually isometric.
\end{proof}

\begin{corollary} \label{cor:rigid finite index}
If $X$ is a compatible hyperbolic structure, then
\[ [\Mod_{0}(X):\SHo(S)] < \infty.\]
\end{corollary}

\begin{proof}
Since $S$ has finitely many ends, it suffices to prove that $\PSHo(S)$ has finite index in $\PMod_{0}(X)$.  Suppose $f_1,f_2,f_3,\ldots \in \PMod_{0}(X)$ is any infinite set of elements.  To prove the index is finite, we must show that there are two 
of these elements that lie in the same coset of $\PSHo(S)$.
Let $\Phi$ be the flux homomorphism described in \Cref{ssec:shogrps}.
By precomposing with elements of $\PSHo(S)$, we can assume that every component of $\Phi(f_j)$ is less than the genus of the piece, for all $j$.  In particular, for any piece $Y$ in a combinatorial neighborhood of an end, $f_j(Y) \cap Y \neq \emptyset$.

By \Cref{thm:rigid to tight}, every $f_j$ is eventually isometric.  Given an end $\E$ and positive integer $N$, there is a compatible neighborhood $U_\E$ of $\E$ so that $f_1,\ldots,f_N$ all restrict to isometric embeddings of $U_\E$ into $X$.  Fixing any piece $Y_\E \subset U_\E$, since $f_j(Y_\E) \cap Y_\E \neq \emptyset$, if $N$ is sufficiently large it follows that the there are $i \neq j$, with $i,j \leq N$, so that the restriction of $f_i$ and $f_j$ to $Y_\E$ agree.  But because these are isometric embeddings on $U_\E$, $f_i$ and $f_j$ agree on $U_\E$.  By taking $N$ even larger, we can assume this holds for some $i \neq j$ and {\em every} end.  But then $f_if_j^{-1}$ is compactly supported, hence in $\PSHo(S)$.
\end{proof}

\section{More general hyperbolic structures}\label{sec:genhypbp}

We continue to consider a fixed rigid structure $\cR$ on $S$, and in this section we introduce a weaker notion of having ``nice'' geometry with respect to $\cR$. We then prove a number of structural consequences of this definition that will be used to prove \Cref{thm:mainbddpieces} in \Cref{sec:bpandsho}.

\subsection{Bounded pieces}\label{ssec:bddpieces}
Let $X$ be a hyperbolic Riemann surface structure on $S$. Throughout this section, we will assume that the boundary components of the pieces of $\cR$ are geodesics with respect to $X$. 

Given a connected subsurface $\Sigma \subset X$ (with geodesic boundary), it will be important to clarify the distinction between the {\em intrinsic} geometry of $\Sigma$, by which we mean the geometry of the induced {\em path} hyperbolic metric on $\Sigma$, and the {\em extrinsic} geometry of $\Sigma$, by which we mean the {\em subspace} metric (which need not be a path metric as minimizing geodesic segments in $X$ between points in $\Sigma$ may not lie in $\Sigma$). 

\begin{definition}
    A hyperbolic structure $X$ on $S$ has \emph{$D$-bounded pieces} for some $D>0$ if the hyperbolic diameter of each piece in $S$, measured as a subset of $X$, is bounded above by $D$. 
    \end{definition}

Whenever the constant $D$ is not relevant to the discussion, we will omit it and simply say that $X$ has {\em bounded pieces}. 
We emphasize that having bounded pieces is an {\em extrinsic} geometric property.
In \Cref{sec:examples}, we will construct various examples of surfaces with bounded pieces.  Some of these examples have pieces with unbounded {\em intrinsic} diameter.

We will prove a few preliminary lemmas highlighting some basic consequences of having bounded pieces. First, recall that a hyperbolic surface $Z$ (possibly with non-empty boundary), has (intrinsic) {\em $B$-bounded geometry} if the following two conditions hold: 

\begin{itemize}
    \item Through every $z\in Z$ there exists a homotopically nontrival closed curve of length at most $B$; 
    \item The infimum of the lengths of closed geodesics is greater than $\frac{1}{B}$.
\end{itemize}

We remark that, in the case when the surface has no boundary or cusps, having $B$--bounded geometry for some $B$ is equivalent to requiring that the injectivity radius of $X$ is uniformly bounded above and below at every point.  When $B$ is unimportant, we say that $Z$ has bounded geometry.

\begin{lemma}\label{lem:bddpiecesbddgeometry}
 If $X$ has bounded pieces, then it has bounded geometry.
\end{lemma}
\begin{proof} Suppose first that there were arbitrarily short geodesics $\{\alpha_k\}$ on $X$. These have collar neighborhoods $\{N(\alpha_k)\}$ with arbitrarily large diameter so that $N(\alpha_k)$ is disjoint from any simple closed geodesic that $\alpha_k$ fails to intersect.  If infinitely many $\alpha_k$ are contained in  pieces (possibly as a boundary components), then at least half the collars are contained in the pieces, violating the assumption of bounded pieces.  The alternative is that infinitely many $\alpha_k$ cut across pieces, again violating the assumption of bounded pieces. Thus there is a lower bound to the length of closed geodesics.

Next, since the diameter of every piece is uniformly bounded above by some $D>0$, the $D$--neighborhood of any point in a piece of $X$ cannot be simply connected. In particular, the injectivity radius at such a point must be less than $D$.  Since the complement of the union of all the pieces is compact, we obtain a uniform upper bound on the injectivity radius at any point of $X$, and we are done.  
\end{proof}

\subsection{Envelopes and compact exhaustions}\label{ssec:envelopes}

From now on we assume that $X$ has $D$--bounded pieces. Thus $X$ also has bounded geometry by \Cref{lem:bddpiecesbddgeometry}. In this subsection we will show that while pieces in $X$ may have poorly behaved intrinsic geometry, the extrinsic control coming from $X$ having bounded pieces allows one to ``envelop'' pairs of pieces by slightly larger subsurfaces with well controlled intrinsic geometry. Recall that distinct, intersecting pieces are called adjacent.

\begin{lemma}\label{lem:geombddsupsurface}
     Suppose that $X$ has $D$--bounded pieces.  Then there exist constants $\kappa = \kappa(D) < 0$ and $s = s(D) \geq D >0$ such that for any two adjacent pieces $Y$ and $Y'$, there exists an embedded subsurface $Z \subset X$ with the following properties: 
     \begin{enumerate}
         \item $Z$ has geodesic boundary;
         \item $\chi(Z) \geq \kappa$;
         \item $Z$ has (intrinsic) $s$-bounded geometry; 
         \item $Y \cup Y' \subset Z$;
         \item $Z$ is contained in the $s$--neighborhood of $Y \cup Y'$
     \end{enumerate}

\end{lemma}
We will call a surface $Z$ as in the lemma an {\em envelope of $Y \cup Y'$}. Given a piece $Y$, an envelope of $Y$ is an envelope of $Y \cup Y'$, where $Y'$ is adjacent to $Y$ and contained in the neighborhood of the end defined by $Y$.  Observe that the diameter of an envelope is bounded by $2D+2s$.

\begin{proof} \Cref{lem:bddpiecesbddgeometry} provides a lower bound on the length of the shortest geodesic of $X$, hence also in any subsurface $Z \subset X$.  

Given a pair $Y,Y'$ of adjacent pieces, first observe that $Y \cup Y'$ is a compact, connected surface with geodesic boundary.  Let $x \in Y \cap Y'$ be any point, and observe that $Y \cup Y'$ is contained in $N_D(x)$, the $D$--ball about $x$.  Let $\widetilde Z_0 \subset \hyp$ be a component of the preimage of $N_D(x)$ in the universal cover $\hyp$ of $X = \hyp/\Gamma$, where $\Gamma$ is the Fuchsian group uniformizing $X$, and let $\widetilde x \in \widetilde Z_0$ be a point in the preimage of $x$. 

Let $G < \Gamma$ denote the stabilizer of $\widetilde Z_0$ in $\Gamma$, and let $g_1,\ldots,g_k \in \Gamma$ be those elements so that $N_D(\widetilde x) \cap g_i N_D(\widetilde x) \neq \emptyset$.  We observe that $G$ is generated by $g_1,\ldots,g_k$, and
\[ \widetilde Z_0 = G \cdot N_D(\widetilde x).\]
From the discreteness of $\Gamma$ and the uniform lower bound on the length of the shortest geodesic in $X$, we deduce that $k$ is bounded by a function of this lower bound, and hence by a function of $D$ (and the rigid structure), by \Cref{lem:bddpiecesbddgeometry}.

Since $\widetilde Z_0$ contains a component of the preimage of $Y$, $G$ contains a conjugate of $\pi_1Y$, so is nonelementary.
The limit set of $G$ is also the accumulation set of $\widetilde Z_0$, and $G$ is nonelementary, it follows that no two components of the preimage of $N_D(x)$ in $\hyp$ can have the same accumulation set.  Therefore, $G$ is the stabilizer of its limit set.  Let, $\widetilde Z_1$ denote the convex hull of the limit set of $G$, and $Z_1 = \widetilde Z_1/G$.  

We claim that the interior of $Z_1$ embeds into $X$.  To see this, we note that since $G$ is the stabilizer of $\widetilde Z_1$, the only way the interior of $Z_1$ can fail to embed is if some element $g \in \Gamma \setminus G$ has $g \cdot \mbox{int}(\widetilde Z_1) \cap \mbox{int}(\widetilde Z_1) \neq \emptyset$.  Since $g$ cannot preserve the limit set, there are a pair of points in the accumulation set of $\widetilde Z_1$ and a pair of points in the accumulation set of $g \cdot \widetilde Z_1$ that link each other, from which we can conclude that $\widetilde Z_0$ and $g \cdot \widetilde Z_0$ intersect.  But then $\widetilde Z_0 = g \cdot \widetilde Z_0$ and hence $\widetilde Z_1 = g \cdot \widetilde Z_1$, a contradiction.  Therefore, the interior of $Z_1$ embeds into $X$.  We let $Z \subset X$ be the image of $Z_1$, which is obtained from $Z_1$ by identifying any boundary components that map to the same curve in $X$.

Since $Y \cup Y' \subset N_D(x)$, there is a component $\widetilde{Y \cup Y'}$ of the preimage of $Y \cup Y'$ which is contained in $\widetilde Z_0$.  Moreover, this component is convex, as  $Y \cup Y'$ has geodesic boundary, and hence $\widetilde{Y \cup Y'} \subset \widetilde Z_1$, hence $Y \cup Y' \subset Z_1 \subset Z$.  Additionally, we have $\chi(Z) = \chi(Z_1)= 1-\mbox{rk}(G) \geq 1-k$.  We further note that $\pi_1Z_1 = G$ is generated by elements translating a basepoint a distance at most $2D$, hence $Z_1$ has intrinsic diameter bounded by a function of $D$, thus so does $Z$, proving that $Z$ has $s$--bounded geometry for some $s \geq D$ depending only on $D$ and the rigid structure. 
\end{proof}

We will use the pigeonhole principle at various points in the proofs that follow, and the following simple fact will be useful in doing so.

\begin{lemma}\label{lem:curvecounting}
Suppose $X$ has $D$-bounded pieces. Then there exist $N,R >0$, depending only on $D$, such that if $Z$ is an envelope of a pair of adjacent pieces in $X$, then the following holds: 
\begin{itemize}
    \item There exists a closed geodesic $\gamma$ in $Z$, of length at most $R$, such that $\gamma$ fills Z. Moreover, any homeomorphism of $Z$ that preserves $\gamma$ is isotopic to the identity. 
    \item The number of curves in $X$ of length at most $R+1$ that intersect $Z$ is bounded above by $N$.\end{itemize}
\end{lemma}

\begin{proof} The first statement is straightforward from \Cref{lem:geombddsupsurface}.  The second statement follows from the fact that $Z$ has diameter bounded above by $2s + 2D$ (\Cref{lem:geombddsupsurface}) and $X$ has bounded geometry (\Cref{lem:bddpiecesbddgeometry}).
\end{proof}

The next two corollaries allow us to use the envelope construction of \Cref{lem:geombddsupsurface} in order to find a ``well-behaved'' compact exhaustion of $X$. We note that the rest of this subsection is not actually necessary for the proof of \Cref{thm:mainbddpieces}. We include it so as to provide a better understanding of the bounded pieces condition.

\begin{lemma} \label{lem:good gut-offs}
    Suppose that $X$ has $D$--bounded pieces. There exists a $B = B(D)>0$ and $\eta = \eta(D)<0$ so that for each end $e$ there exists a neighborhood basis $\{V_i\}$ for $e$ so that for all $i$ 
    \begin{itemize}
        \item $V_{i+1} \subset V_i$,
        \item each curve in $\partial V_i$ has length bounded above and below by $B$ and $\frac{1}{B}$, respectively, and
        \item the surface cobounded by $\partial V_i$ and $\partial V_{i+1}$ has Euler characteristic bounded below by $\eta$. 
    \end{itemize}
\end{lemma}
\begin{proof}
    Let $U = Y_1 \cup Y_2 \cup \ldots$ be a combinatorial neighborhood of $e$, and let $Z_j$ be the envelope of $Y_j$ and $Y_{j+1}$, also assumed to have geodesic boundary.  For each $j$, we let $\alpha_j = Y_j \cap Y_{j+1}$ be the (geodesic) curve where these pieces meet, which is contained in $Z_j$, and cuts off a neighborhood of $e$. Each $Z_j$ has a submulticurve in its boundary $A_j \subset \partial Z_j$ such that $A_j$ cuts off a neighborhood of $e$ and the interior of $Z_j$ is contained in the complement.
    
    By \Cref{lem:curvecounting}, the interior of $Z_j$ can intersect at most $N$ other envelopes.  We will choose a subsequence $\{Z_{n_i}\}$ of $\{Z_j\}$ inductively, and define $V_i$ to be the neighborhood of $e$ cut off by $A_{n_i}$. To define the subsequence, first set $Z_{n_1} = Z_1$.  Having chosen $Z_{n_{i-1}}$, choose $Z_{n_i} \subset V_{i-1}$ with interior disjoint from $Z_{n_{i-1}}$ and with $0 < n_i - n_{i-1} \leq N +2$.  To see that this is possible, we first observe that $Z_{n_{i-1}+1},\ldots,Z_{n_{i-1}+N+2}$ are $N+1$ envelopes, hence one of these must be disjoint from the interior of $Z_{n_{i-1}}$, and we pick one of these to be $Z_{n_i}$. Clearly $0 < n_i - n_{i-1} \leq N+2$. To see that $Z_{n_i} \subset V_{i-1}$, observe that $\alpha_{n_{i-1}} \subset Z_{n_{i-1}}$, while $\alpha_{n_i} \subset Z_{n_i}$.  Since $\alpha_{n_i}$ is contained in the neighborhood of $e$ cut off by $\alpha_{n_{i-1}}$, and $\alpha_{n_i}$ is in the complement of $Z_{n_{i-1}}$, we must have $\alpha_{n_i} \subset V_{i-1}$.  But then, $Z_{n_i} \subset V_{i-1}$, as required.

    By construction we have $V_{i+1} \subset V_i$.  Since $\partial V_i = A_{n_i}$, and $A_{n_i}$ is a union of boundary components of $Z_{n_i}$ which has $s$--bounded geometry, with $s = s(D)$, and $\chi(Z_{n_i}) \geq \kappa = \kappa(D)$, it follows that the length of each component of $A_{n_i}$ is bounded above by $B$ and below $\frac1B$, for some $B = B(D)$.
    
    Finally, observe that the subsurface $W_i$ bounded by $\partial V_i$ and $\partial V_{i+1}$ is contained in the union
    \[ W_i \subset Z_{n_i} \cup Y_{n_i+1} \cup \cdots \cup Y_{n_{i+1}} \cup Z_{n_{i+1}}.\]
    If we let $g$ denote the genus of each piece, so that $\chi(Y_j) = 2-2g-2 = -2g$, for all $j$, then
    \[\chi(W_i) \geq \chi(Z_{n_i}) + \chi(Z_{n_{i+1}}) - |n_{i+1}-n_i|2g \geq 2\kappa - (N+2)2g.  \]
    The right-hand side provides the required lower bound $\eta = \eta(D) < 0$ on $\chi(W_i)$, completing the proof.
\end{proof}

\begin{corollary} \label{cor:bounded exhaustion}
    Suppose that $X$ has $D$--bounded pieces. There exists a $B=B(D) >0$ and $\eta = \eta(D) <0$ and a compact exhaustion $\{C_i\}$ of $X$ so that for all $i$
    \begin{itemize}
        \item $C_1$ separates the ends of $X$,
        \item $\partial C_i$ is contained in the interior of $C_{i+1}$, 
        \item the length of each curve in $\partial C_i$ is bounded above and below by $B$ and $\frac{1}{B}$, respectively, and
        \item the surface cobounded by $\partial C_i$ and $\partial C_{i+1}$ has Euler characteristic bounded below by $\eta$. 
    \end{itemize}
\end{corollary}
\begin{proof}
    We can apply \Cref{lem:good gut-offs} to each end $e$ of $X$ to find neighborhood bases $\{V_i^e\}_{i=1}^{\infty}$ and constants $B>0$ and $\eta < 0$. The construction in the proof ensures that for distinct ends $e,e'$ and any $i \geq 1$, we have $V_i^e \cap V_i^{e'} = \emptyset$, and we set  $C_i = X\setminus \left(\bigcup_{e} V_i^e\right)$. 
\end{proof}

A {\em $D$--bounded pants decomposition} is a pants decomposition where every component has length bounded above by $D$ and below by $\frac1D$.  A surface with bounded geometry and bounded Euler characteristic has a bounded pants decomposition, with bounds depending only on lower bounds on the Euler characteristic and upper bounds on the length of the boundary curves; see \cite{Parlier2023}. Applying this to each of the subsurfaces $C_{i+1} \setminus C_i$ from \Cref{cor:bounded exhaustion} gives the following.

\begin{corollary}\label{cor:bdd pants}
    If $X$ has bounded pieces, then it has a bounded pants decomposition.
\end{corollary}

Note that the condition of having a bounded pants decomposition is strictly stronger than having bounded geometry. Examples in the planar surface case are constructed in \cite{Kinjo2011} and in the non-planar case in \cite{BPV2024}. 

\subsection{Propagating towards ends}\label{ssec:propogation}

We will ultimately need to decide when two asymptotically conformal maps agree on the entire neighborhood of an end; or equivalently, to decide when an element restricts to the identity on a neighborhood of an end.  The next lemma utilizes \Cref{lem:curvecounting} to provide a mechanism for doing that, allowing us to simply check that an element is the identity on a piece sufficiently far out in the end. Recall that by \Cref{cor:nearly isometric}, for any $K>1$, for $f \in \Mod_0(X)$ and any end, there is a neighborhood $V$ of the end on which $f$ is $K$--tight.

\begin{lemma} \label{lem:identity propogates} Suppose that $X$ has $D$--bounded pieces.  Let $K > 1$ be such that 
\[ R K^N \leq R+1,\]
where $R$ and $N$ are as in \Cref{lem:curvecounting}.  Suppose $f \in \Mod_0(X)$, $U$ is a combinatorial neighborhood of an end, $s >0$ is as in \Cref{lem:geombddsupsurface}.  If $f$ is $K$--tight on the $(s+R+1)$--neighborhood of $U$ and is the identity on the boundary piece $Y$ of $U$, then (up to isotopy) $f$ restricts to the identity on $U$.
\end{lemma}
\begin{proof}   Suppose $U = Y_1\cup Y_2 \cup \ldots$ is a combinatorial neighborhood of an end with boundary piece $Y = Y_1$. Let $f \in \Mod_0(X)$ be $K$--tight on the $s+R+1$--neighborhood of $U$ and the identity on $Y$.  Let $Z$ be an envelope of $Y \cup Y'$ from \Cref{lem:geombddsupsurface}, where $Y'= Y_2$ is the piece adjacent to $Y$ in $U$.

Let $\gamma$ be a closed geodesic in $Z$ of length at most $R$ that fills $Z$, and $\Omega$ the set of curves in $X$ of length at most $R+1$ that meet $Y$.  By \Cref{lem:curvecounting}, since $Y \subset Z$ we know $|\Omega| \leq N$.
\begin{claim}
    For all $k \leq N$,  $f^k(\gamma) \in \Omega$. 
\end{claim}
\begin{proof}[Proof of claim] Since $f$ restricts to the identity on $Y$, it follows that $f^k(\gamma)$ intersects $Y$ for all $k$.  We first prove by induction that for all $1 \leq k \leq N$, we have
\[ \ell(f^k(\gamma)) \leq K^k \ell(\gamma). \]
For this, we first note that $\gamma$ is contained in $Z$, and hence in the $s$--neighborhood of $Y \cup Y'$, and thus also in the $s$--neighborhood of $U$. Therefore
\[ \ell(f(\gamma) \leq K \ell(\gamma),\]
since $f$ is $K$--tight on the $s$--neighborhood of $U$, thus proving the statement for $k=1$.  Now assume
\[ \ell(f^k(\gamma)) \leq K^k \ell(\gamma)\]
for some $1 \leq k < N$.  Our assumption on $K$ and the fact that $\ell(\gamma) \leq R$ implies
\[ \ell(f^k(\gamma)) \leq K^k R < K^NR \leq R+1.\]
Since $f^k(\gamma)$ intersects $Y$, we see that $f^k(\gamma)$ is contained in the $R+1$--neighborhood of $Y$, and hence  also contained in the $s+R+1$--neighborhood of $U$.  Therefore
\[ \ell(f^{k+1}(\gamma)) \leq K \ell(f^k(\gamma)) \leq K^{k+1}\ell(\gamma), \]
completing the induction.  Our choice of $K$ implies
\[ \ell(f^k(\gamma)) \leq K^k \ell(\gamma) \leq K^NR \leq R+1.\] 
Therefore, $f^k(\gamma) \in \Omega$, for all $1 \leq k \leq N$, proving the claim.
\end{proof}

Now observe that
\[ \gamma, f(\gamma),f^2(\gamma),\ldots,f^N(\gamma) \in \Omega\]
and so by the pigeonhole principle, this list must contain two curves which are the same, say $f^m(\gamma) = f^{m+k}(\gamma)$ for some $0 \leq m < m+k \leq N$.  Hence $f^k(\gamma) = \gamma$ and $k \leq N$.
Thus,
\[ \Omega_0 = \{\gamma,f(\gamma),\ldots, f^k(\gamma) \} \subset \Omega\] is a finite, $f$--invariant set of curves.  Since $f^k(\gamma) = \gamma$, it follows that $f^k(\delta) = \delta$ for all $ \delta \in \Omega_0$.
Let $W$ be the subsurface filled by the elements of $\Omega_0$.  We observe that $W$ contains $Z$, hence $Y$, and is $f$--invariant since $\Omega_0$ is.

The restriction of $f$ to $W$ thus has finite order, up to isotopy.  Since $f$ is the identity on the subsurface $Y$, it is necessarily isotopic to the identity on all of $W$.  In particular, $f$ is isotopic to the identity on $Y \cup Y'$.

Now we may repeat the above argument, replacing $U$ with the combinatorial neighborhood $U' = Y_2 \cup Y_3 \cup \ldots$ having $Y' = Y_2$ as boundary piece, an deduce that $f$ is also the identity on $Y' \cup Y''$, where $Y''=Y_3$, and we deduce that $f$ is isotopic to the identity on $Y_1 \cup Y_2 \cup Y_3$.  Continuing inductively, we see that $f$ is in fact the isotopic to the identity on the entire original combinatorial neighborhood $U$, proving the lemma.
\end{proof}

\subsection{Finite families of maps}\label{ssec:compactness}

In this final subsection we combine the above in order to gain control over sufficiently large finite families of asymptotically conformal maps with flux zero. Note that by \Cref{cor:nearly isometric}, given a family of asymptotically conformal maps one can obtain control over lengths of curves living in small enough end-neighborhoods. However, if we add the condition that all of our maps have zero flux we can obtain additional control on where the images of curves lie. 

\begin{lemma} \label{lem:controlled drift}
    Suppose $X$ has $D$--bounded pieces and for $j=1,\ldots,M$ let $f_j \in \Mod_0(S)$ with flux zero, i.e.~$\Phi(f_j) = 0$. For any $K>1$, $r >0$, and any end $e$ of $S$, there exists a pair of combinatorial neighborhoods $U \subset V$ of $e$ such that $N_r(U) \subset V$, $f_j$ is $K$-tight on $V$ for each $j$, and if $\gamma$ is a closed geodesic in $N_r(U)$, then (the geodesic representative of) $f_j(\gamma)$ lies in $V$ for all $j$.

\end{lemma}
\begin{proof}
By shrinking $K > 1$ in statement we may also choose $\epsilon > 0$ such that
    \[ K(r+D) + \epsilon < r+D+1,\]
and so that a $K$--local $(K,0)$--quasi-geodesic curve is within $\epsilon$ of the geodesic it is homotopic to.

By \Cref{thm:EMS} we can choose asymptotically isometric representatives for each of the $f_j$. After taking intersections of neighborhoods, there is a combinatorial neighborhood $V$ of the end $e$ so that for all $j$, $f_j$ is $K$--biLipschitz on $V$. Applying \Cref{cor:nearly isometric} and potentially shrinking $V$ we can also assume that each $f_j$ is $K$--tight on $V$. 

Now let $V$ be such a neighborhood, and $U$ a combinatorial neighborhood with $N_{r+D+1}(U) \subset V$.
Because $f_1,\ldots,f_M$ are $K$--biLipschitz on $V$, hence  on $U$, we get that for any piece $Y$ in $U$, the diameter of $f_i(Y)$ is at most $KD$, for all $j$. 
Fix a $j=1,\ldots,M$.
Since $f_j$ has flux zero for any piece $Y$, we have $f_j(Y)\cap Y \neq \emptyset$.  Consequently, for any $x \in U$, if $Y$ is a piece in $U$ that contains $x$, we see that
\[ d(f_j(x),U) \leq d(f_j(x),Y) \leq \mbox{diam}(f_j(Y)) \leq KD.\]
Any point $x \in N_r(U)$ which is not in $U$ is within distance $r$ from a point $y \in U$, and thus
\[ d(f_j(x),U) \leq d(f_j(x),f_j(y)) + d(f_j(y),U) \leq Kr + KD. \]
Now if $\gamma$ is a closed geodesic in $N_r(U)$, then the inequality above implies
\[ f_j(\gamma) \subset N_{Kr+KD}(U).\]  
The geodesic representative of $f_i(\gamma)$ is within $\epsilon$ of $f_j(\gamma)$, and so from our choice of $K$ and $\epsilon$ we see that the geodesic representative of $f_j(\gamma)$, is contained in $N_{K(r+D)+\epsilon}(U) \subset N_{r+D+1}(U) \subset V$
\end{proof}

Finally we combine this with \Cref{lem:curvecounting} and \Cref{lem:identity propogates} in order to see that when the finite collection of mapping classes is large enough, there must be two that agree on small enough neighborhoods of the ends of $X$. 

\begin{lemma}\label{lem:creatingcompactsupport}
    Suppose that $X$ has $D$--bounded pieces and let $N>0$ be as in \Cref{lem:curvecounting}. If $\{f_j\}_{j=1}^{M}$ is a family of flux zero maps in $\Mod_0(X)$ with $M \geq N^n + 1$, then there exists a compact subsurface $\Sigma \subset X$ so that there exists some $j\neq j'$ with $f_{j'}^{-1}f_{j}$ supported on $\Sigma$.
\end{lemma}

\begin{proof} 
Let $M \geq N^n+1$ and $\cF = \{f_j\}_{j=1}^{M}$ be a family of flux zero maps in $\Mod_0(X)$.

Let $K > 1$ be such that
\[ R K^{2N} \leq R+1,\]
where $R,N$ are as in \Cref{lem:curvecounting}. Enumerate the $n$ ends, and for each one we apply \Cref{lem:controlled drift} to the family $\cF$ with the chosen $K$ and with $r=s+R+1$ to find a pair of combinatorial neighborhoods $U_k \subset V_k$ of the $k^{th}$ end, for each $k=1,\ldots,n$, satisfying the conclusion of the lemma.  Now set
\[ \Sigma = \overline{X \setminus \bigcup_{k=1}^n U_k}.\]

For each $U_k$, let $Y_k$ be its boundary piece and $Y_k'$ the adjacent piece in $U_k$.  Let $Z_k$ be an envelope of $Y_k \cup Y_k'$.  Choose a curve $\gamma_k$ of length at most $R$ in $Z_k$ that fills $Z_k$, as in \Cref{lem:curvecounting}, and let $\Omega_k$ be the set of curves of length at most $R+1$ that intersect $Z_k$.  By the same lemma, $|\Omega_k| \leq N$.

For any $k$ and $j$, since $\gamma_k \in U_k \subset V_k$, we have
\[ \ell(f_j(\gamma_k)) \leq K \ell(\gamma_k) \leq KR < R+1,\]
by our choice of $K$, and since $f_j$ is $K$--tight on $V_k$.  Additionally, since each $f_j$ has flux zero, we get that $f_j(Y_k) \cap Y_k \neq 0$, for all $j$ and $k$.  Consequently, $f_j(\gamma_k) \in \Omega_k$, for all $j$ and $k$.
By the pigeonhole principle, there exists $j' \neq j$ so that
\[(f_{j'}(\gamma_1),\ldots,f_{j'}(\gamma_n)) = (f_j(\gamma_1),\ldots,f_j(\gamma_n)) \in \Omega_1 \times \ldots \times \Omega_n,\]
since $|\Omega_1 \times \ldots \times \Omega_n| \leq N^n < N^n+1 \leq M$.
Therefore, $f_{j'}^{-1}f_j(\gamma_k) = \gamma_k$ for all $k$. But \Cref{lem:curvecounting} implies $f_{j'}^{-1}f_j$ is (isotopic to) the identity on $Y_k$ for each $k$.

Next by \Cref{lem:controlled drift}, if $\alpha$ is a closed geodesic in $N_{s+R+1}(U_k)$ for any $k$, then $f_j(\alpha) \subset  V_k$ for all $j$, and hence
\[ \ell(f_{j'}^{-1}(f_j(\alpha))) \stackrel{K}{\asymp} \ell(f_j(\alpha)) \stackrel{K}{\asymp} \ell(\alpha).\]

Thus, $f_{j'}^{-1}f_j$ is $K^2$--tight on the $s+R+1$--neighborhood of $U_k$ for each $k$ and is the identity on each $Y_k$.  Since
\[ R(K^2)^N = RK^{2N} \leq R+1,\]
\Cref{lem:identity propogates} implies $f_{j'}^{-1}f_j$ is the identity on $U_k$ (up to isotopy) for all $k$.  Thus, $f_{j'}$ and $f_j$ are supported in $\Sigma$, as required.
\end{proof}

\section{Bounded pieces and surface Houghton groups}\label{sec:bpandsho}

In this section we prove \Cref{thm:mainbddpieces}; that is, if $X$ is a hyperbolic Riemann surface structure on $S$ that has bounded pieces for the rigid structure $\cR$, then
\[ [\Mod_0(X):\SHo(S,\cR)] < \infty,\]
provided $\SHo(S,\cR) < \Mod_0(X)$.

We fix the rigid structure $\cR$ on $S$ throughout the rest of this section and as usual write $\SHo(S,\cR) = \SHo(S)$.
\begin{proof}[Proof of \Cref{thm:mainbddpieces}]
     Since $S$ has finitely many ends, the pure subgroup $\PSHo(S)<\SHo(S)$ has finite index in $\SHo(S)$, so it suffices to prove that $\PSHo(S)$ has finite index in $\PMod_0(X)$.  Since the image of the flux homomorphism, $\Phi(\PSHo(S))$, is contained in $(g \mathbb Z)^n$, where $g$ is the genus of every piece of $\cR$, and since
    \[ \PMod_0^g(X) := \Phi^{-1}((g \mathbb Z)^n) \cap \PMod_0(X) < \PMod_0(X)\]
    is a finite index subgroup, it in fact suffices to prove that $\PSHo(S)$ has finite index in $\PMod_0^g(X)$.

\begin{claim}
    $[\PMod_0^g(X): \PSHo(S)] <M$, where $M = N^n + 1$ and $N$ is as in \Cref{lem:curvecounting}.
\end{claim}
\begin{proof}[Proof of claim]
    Suppose $f_1,f_2,f_3,\ldots,f_M \in \PMod_0^g(X)$ and $M = N^n+1$.  By composing with elements of $\PSHo(S)$, we can assume that all $f_i$ have flux zero, that is, $\Phi(f_i) = 0$.  We must show that for some $i \neq j$, $f_i$ and $f_j$ are in the same coset of $\PSHo(S)$.

    We can now immediately apply \Cref{lem:creatingcompactsupport} to the family $\{f_j\}_{j=1}^{M}$. This gives us $j \neq j'$ so that $f_{j'}^{-1}f_{j}$ are compactly supported. Thus they are in the same coset of $\PSHo(S)$, as required. 
\end{proof}
This completes the proof of the theorem.
\end{proof}

\subsection{End-Periodic Maps}\label{ssec:endperiodic}

The proof of \Cref{thm:mainbddpieces} is easily modified to deduce a similar conclusion under the weaker assumption that there is a finite index subgroup of $\SHo(S)$ contained in $\Mod_0(X)$.  In fact, we can assume even less, as we now explain.

In this subsection we prove \Cref{thm:mainep}, which asserts that the presence of an element of $\mathcal H(S)$ that is both end-periodic and asymptotically conformal is enough to guarantee that $\mathcal H(S)$ has finite index in $\Mod(X)$. 
We refer the reader to \cite{FenleyCT,Fenley-depth-one,CC-book} for more background on end-periodic homeomorphisms. 

\begin{definition}
    An \emph{end-periodic homeomorphism} of $S$ is a map satisfying the following. There exists an $m \geq 1$ so that for each end $e$ of $S$, there exists a neighborhood $U_e$ of $e$ in $S$ so that 
    \begin{enumerate}
        \item $f^m(U_e) \subsetneq U_e$ and $\{f^{nm}(U_e)\}_{n>0}$ forms a neighborhood basis of e, or
        \item $f^{-m}(U_e) \subsetneq U_e$ and $\{f^{-nm}(U_e)\}_{n>0}$ forms a neighborhood basis of e.
    \end{enumerate}
    We say that a mapping class is end-periodic if it has an end-periodic representative. 
\end{definition}

Note that if $f \in \SHo(S)$ is end-periodic, then we may replace the $U_e$ appearing in the definition with combinatorial neighborhoods.  

\begin{proof}[Proof of \Cref{thm:mainep}]
    Let $\cR$ be the rigid structure defining $\SHo(S)$. We first pass to a power of $f$ so that $f \in \PSHo(S)$; abusing notation, we still call this map $f$ and we may assume that $m=1$ in the definition of end-periodicity. Fix an enumeration of the ends of $S$ as $e_1,e_2,\ldots,e_n$ and for each $i=1,\ldots,n$, let $\Phi_i(f)$ be the $i$-th coordinate of the flux homomorphism. That is, $\Phi_i(f)$ exactly measures the genus of the subsurface $U_{e_{i}} \setminus f^{\pm}(U_{e_i})$.

    Set $M = \lcm(\Phi_{1}(f),\ldots,\Phi_{n}(f))$ and let $\cR'$ be the new rigid structure obtained by gluing together $M$-pieces of $\mathcal R$  to form the new pieces and let $\SHo'(S)$ be the corresponding surface Houghton group. By \cite[Theorem 1.1(1)]{AraDomLei}, $\SHo'(S)$ and $\SHo(S)$ are commensurable. Additionally, since $f \in \SHo(S)$, there is some defining surface $Z$ outside of which $f$ is $\cR$-rigid. Since $\cR'$ is obtained by gluing pieces of $\cR$ together, there is some potentially larger subsurface $Z'$ outside of which $f$ is $\cR'$-rigid. 

    Now let $g \in \PSHo'(S)$. Take a large suited surface $Z_g$ that contains $Z'$ and outside of which $g$ is $\cR'$-rigid. Write $U_1,\ldots,U_n$ for the components of $S \setminus Z_g$ that see each end $e_1,e_2,\ldots,e_n$. We verify that $g$ is asymptotically conformal by showing that it is on each subsurface $U_i$. 

    For any fixed $i$ we have that $\Phi_i(g) = m_iM$ for some $m_i \in \mathbb Z$. By the definition of $M$, there exists some $k_i \in \mathbb Z$ so that $k_i \Phi_i(f) = m_i M$. Thus both $g$ and $f^{k_i}$ are $\cR'$-rigid on $U_i$ and hence agree. Therefore, $g$ is asymptotically conformal on $U_i$, since $f^{k_i}$ is. Finally, as $Z_g$ is a compact subsurface, we conclude that $g \in \Mod_0(X)$.  Therefore, $\PSHo'(S) < \Mod_0(X)$.  Because $X$ had bounded pieces for $\cR$ it also has bounded pieces for $\cR'$.  Thus by \Cref{thm:mainbddpieces}, $\PSHo'(S)$ has finite index in $\Mod_0(X)$.  Since $\PSHo'(S)$ has finite index in $\PSHo(S)$, it follows that $\SHo(S) \cap \Mod_0(X)$ has finite index in both $\SHo(S)$ and $\Mod_0(X)$, as required.
\end{proof}

\section{$L^{p}$ quasiconformal deformation space}\label{sec:lp}
As mentioned in the introduction, we now translate our results to the $L^{p}$ quasi-conformal deformation spaces $\T^{p}(X)$, for $p \geq 2$; these are also known as $p$-integrable Teichm\"{u}ller spaces. First we need to introduce some background on these spaces. We direct the reader to \cite{Yanagishita2014,Yanagishita2017} for more details. 

\begin{remark}
    We fix $p\geq2$ throughout this section. Some of the results of Yanagashita from \cite{Yanagishita2014} still hold when $p \geq 1$ and we direct the reader to that paper for more details. However, the condition that $p\geq 2$ is necessary for proving that the Bers projection is continuous. See \cite[Proposition 3.6]{Yanagishita2014} and the remark that immediately follows for more details on this. When one considers the \emph{universal} $L^{p}$ Teichm\"{u}ller space, i.e. $\T^{p}(D)$, one can still prove that the Bers embedding is continuous and equip this space with the structure of a complex Banach manifold even for $1\leq p < 2$ \cite{WM2023}.
\end{remark}

Given a hyperbolic Riemann surface $X$, write $\rho_\hyp(z) |dz| = \frac{|dz|}{2 \Im z}$ for the Poincar\'e metric on $\hyp$, and $\rho_X$ for the descent to $X$.  We write $\rho_X^2$ and $dA_X$ for the associated Riemannian metric and Poincar\'e area form.

For each $p \geq 1$, we let $\CL^p(X)$ denote the space of $p$--integrable $(-1,1)$--forms, that is, $\mu$ such that
\[ \| \mu\|_p = \left( \int_X |\mu|^p \, dA_X \right)^{\frac1p} < \infty. \]

Further, let
\[ \CL^{p,\infty}(X) = \CL^p(X) \cap \CL^\infty(X),\]
with the norm $\| \mu \|_p+\| \mu \|_\infty$.

Following Yanagishita \cite{Yanagishita2014} (see also \cite{Yanagishita2017} for the $p=2$ case), we define the {\em $L^p$ quasiconformal deformation space} (or \emph{$p$-integrable Teichm\"{u}ller space}), $\T^p(X) \subset \T(X)$ to be the subspace consisting of $[\mu] \in \T(X)$ with representative \[\mu \in \Bel^p(X) = \Bel(X) \cap \CL^p(X).\] 
If $f \colon X \to Y$ is quasiconformal and $\mu_f \in \Bel^p(X)$, then we say that $f$ is {\em $p$--integrably quasiconformal}. 
 
If $f \colon X \to Y$ is any quasi-conformal map, we have a natural change of basepoint homeomorphism $R_{f}:\T(X) \to \T(Y)$ given by $R_{f}\cdot [h:X \rightarrow Z] = [h \circ f^{-1}:Y \rightarrow Z]$. Yanagashita proves the following. 

\begin{proposition}\cite[Proposition 5.1]{Yanagishita2014}
    If $\mu_{f} \in \Bel^{p}(X)$, i.e. if $[f:X \rightarrow Y] \in \T^{p}(X)$, then $R_{f}$ is a homeomorphism of $\T^{p}(X)$ onto $\T^{p}(Y)$ with the topologies from the norm $||\cdot||_{p}+||\cdot||_{\infty}$.  
\end{proposition}

This provides us the following reinterpretation which roughly states that $L^{p}$-Teichm\"{u}ller components are either disjoint or coincide. 

\begin{corollary}\label{cor:components}
    Either $R_{f}(\T^{p}(X)) = \T^{p}(Y)$ or $R_{f}(\T^{p}(X)) \cap \T^{p}(Y) = \emptyset$. In particular, the first case occurs exactly when $[f:X \rightarrow Y] \in \T^{p}(X)$. 
\end{corollary}

\begin{proof}
    If $[f:X \rightarrow Y] \in \T^{p}(X)$ then the previous proposition exactly states that $R_{f}(\T^{p}(X)) = \T^{p}(Y)$. 

    Now, suppose that $R_{f}(\T^{p}(X)) \cap \T^{p}(Y) \neq \emptyset$. I.e., there exists some $[h:X \rightarrow Z] \in \T^p(X)$ such that $R_{f}\cdot [h:X\rightarrow Z] \in \T^{p}(X)$. Thus (up to isotopy) the map $h \circ f^{-1}:Y \rightarrow Z$ is $L^{p}$-quasi-conformal. We also know that $h^{-1}:Z \rightarrow X$ is $L^{p}$-quasi-conformal. Finally, we use the fact contained in the proof of \cite[Proposition 5.1]{Yanagishita2014} that the composition of $L^{p}$-quasi-conformal maps is again $L^{p}$-quasi-conformal to conclude that $[f:X \rightarrow Y]\in \T^{p}(X)$ and hence we are in the first case and $R_{f}(\T^{p}(X)) = \T^{p}(Y)$.
\end{proof}

\subsection{$L^{p}$ modular groups}\label{ssec:lpmodgroups}

We define $\Mod^p(X) < \Mod(X)$ to be the subgroup consisting of mapping classes represented by quasiconformal homeomorphisms $f$ such that $\mu_{f} \in \Bel^{p}(X)$. Note that by \Cref{cor:components} this is equivalent to the subgroup of $\Mod(X)$ that stabilizes $\T^{p}(X)$ as a subset of $\T(X)$. We will refer to this group as the \emph{$L^{p}$ modular group}.

Note that if $X$ is compatible (as in \Cref{sec:compatiblestructures}), then any element of $\SHo(S)$ is represented by a quasiconformal homeomorphism $f \colon X \to X$ for which $\mu_f$ is compactly supported.  Consequently, $\SHo(S) < \Mod^p(X)$ for all $p \geq 1$.  \Cref{sec:examples} provides more exotic examples.

\subsection{$p$-integrable to asymptotically conformal}\label{ssec:pintotoac} 

In this section we will see that under certain geometric conditions, we have $\T^{p}(X) \subset \T_{0}(X)$. This is done in \cite[Section 6]{Yanagishita2014}, but for the sake of exposition we give an outline of the ideas here. The key observation is that the $p$-norm provides an upper bound on the $\infty$-norm. However, we first need to introduce spaces of quadratic differentials.

Setting $\hyp^{*}$ to be the lower half-plane and $\rho_{\hyp^{*}}(z)|dz| = \frac{|dz|}{-2 \Im z}$, we denote $X^* = \hyp^*/\Gamma$ to be the conjugate surface to $X = \hyp/\Gamma$. We again write $\rho_{X^*}$ for the descent of $\rho_{\hyp^*}|dz|$ to $X^*$, with $dA_{X^*}$ denoting the associated area form. We then let $A^\infty (X^*)$ be the Banach space of holomorphic, essentially bounded,$(2,0)$-forms (i.e. quadratic differentials) on $X^{*}$. That is, those such $\phi$ with 
\begin{align*}
    ||\phi||_{\infty} = \sup_{z\in X^*} |\phi(z)| \rho_{X^*}(z)^{-2} < \infty,
\end{align*}
For any $X = \hyp/\Gamma$ of the first kind, the Bers embedding defines a complex manifold structure on $\T(X)$ modeled on $A^\infty (X^*)$, e.g. see \cite[Chapter 6]{GardLak}. 

Similarly, we denote by $A^{p}(X^*)$ the Banach space of $p$-integrable, holomorphic, quadratic differentials. That is, those $\phi$ such that 
\begin{align*}
    ||\phi||_{p} = \left(\int_{X^*} |\phi(z)|^{p} \rho_{X^{*}}(z)^{-2p} dA_{X^*} \right)^{\frac{1}{p}} < \infty.
\end{align*}

The key geometric condition that we will use is the following: 

\begin{definition}
    A Riemann surface, $X$, satisfies \textbf{Lehner's condition} if the infimum of the length of all simple closed geodesics on $X$ is positive.
\end{definition}

Equipped with this definition, we now recall a useful result of  Yanagashita \cite[Proposition 4.2]{Yanagishita2014}: 

\begin{lemma}\cite[Proposition 4.2]{Yanagishita2014}\label{lem:Lpestimate}
    Let $X=\hyp/\Gamma$ be a Riemann surface  satisfying Lehner's condition. There exists a constant $C_p > 0$, depending on $p$ and the lower bound given by Lehner's condition, so that for all $\phi \in A^{p}(X^*)$ 
    \begin{align*}
        ||\phi||_{\infty} \leq C_p ||\phi||_p.
    \end{align*}
\end{lemma}

Note that if $X$ has bounded geometry, then $X$ satifies Lehner's condition. The previous lemma implies that for $X$ satisfying Lehner's condition, we have $A^{p}(X^*) \subset A^{\infty}(X^*)$. It is also in the presence of this condition that Yanagashita equips $\T^{p}(X)$ with the structure of a complex Banach manifold modeled locally on $A^{p}(X^*)$ \cite[Theorem 4.4 \& Proposition 5.1]{Yanagishita2014}, with charts are given by the Bers embedding. This is the reason that we define quadratic differentials on $X^*$ as opposed to $X$.

We can also introduce an analog of asymptotic conformality for quadratic differentials. For $\phi \in A^{\infty}(X^*)$, we say that $\phi$ \textbf{vanishes at infinity} if for every $\epsilon > 0$, there exists a compact subset $E \subset X^*$ so that 
\begin{align*}
    ||\phi\vert_{X^* \setminus K}||_{\infty} < \epsilon.
\end{align*}
Denote by $A_0(X^*)$ the subset of bounded holomorphic quadratic differentials that vanish at infinity. Now, $\T_0(X)$ is given a complex manifold structure modeled on the Banach space $A_0(X^*)$ in \cite{EGL2000}. 

Yanagashita makes use of \Cref{lem:Lpestimate} in order to prove the following statement relating $p$-integrable quadratic differentials and those that vanish at infinity.  

\begin{proposition}\cite[Proposition 6.6]{Yanagishita2014}\label{prop:QDLp}
    Let $X = \hyp/\Gamma$ be a Riemann surface satisfying Lehner's condition. Then $A^{p}(X^*) \subset A_{0}(X^*)$. 
\end{proposition}

Given any holomorphic quadratic differential $\phi$ one can form the associated $(-1,1)$-- differential,
\begin{align*}
    \mu[\phi] = \frac{\overline{\phi}}{\rho_{X}^{2}},
\end{align*}
where in terms of universal covers, $\bar \phi$ is given by $\bar \phi(z) = \overline{\phi(\bar z)}$.
Beltrami differentials arising from this construction are often known as \emph{harmonic Beltrami differentials}.
If $\phi \in A^p(X^*)$, then this will be an element of $\CL^{p,\infty}(X)$, provided that we started with a quadratic differential in $A^{p}(X^{*})$.  Indeed, observe that
\[ \|\mu[\phi]\|_p^p = \int_X \Big| \frac{\overline{\phi}}{\rho_X^2} \Big|^p dA_X  = \int_{X^{*}} |\phi|^{p} \rho_{X^{*}}^{-2p} dA_{X^{*}}  = \|\phi\|_p^p \] 
Similarly, given a quadratic differential $\phi \in A_0(X^*)$, then $\mu[\phi]$ will be in $\Bel_0(X)$. 

Finally we can use this to obtain the same statement on the level of Teichm\"{u}ller spaces. 

\begin{lemma}
Let $X = \hyp/\Gamma$ be a Riemann surface satisfying Lehner's condition. For any $[\mu], [\mu'] \in \T^p(X)$, there exists an asymptotically conformal map $f:X^{\mu} \to X^{\mu'}$ such that $f^{\mu'} \simeq f \circ f^\mu$.
\end{lemma}

\begin{proof}
We first consider a single $[\mu] =  [f^\mu: X \rightarrow X^\mu] \in \T^p(X)$. Apply \cite[Section 6.4, Theorems 2 \& 3]{GardLak} in order to obtain a neighborhood of the identity in $\T(X^\mu)$ that is homeomorphic, under the Bers embedding, onto an open subset of $A^\infty((X^\mu)^*)$ contained in the ball of radius $2$ and containing the ball of radius $\frac{2}{3}$. Then \cite[Theorem 4.4]{Yanagishita2014} implies that, when restricted to $\T^p(X^\mu)$, the Bers embedding maps homeomorphically into $A^p((X^\mu)^*)$. Combining these facts yields a neighborhood of the identity in $\T^p(X^\mu)$ that is homeomorphic, under the Bers embedding, to an open subset of $A^p((X^\mu)^*)$ contained in the ball of radius $2$ and containing the ball of radius $\frac{2}{3}$.

Pulling this neighborhood back to $\T^p(X)$ via the map $R_{f^\mu}^{-1}$ gives a neighborhood, $\mathcal{U}([\mu])$, about $[\mu]=[f^{\mu}:X\rightarrow X^\mu]$ in $\T^p(X)$ that is homeomorphic onto an open subset of $A^p((X^\mu)^*)$ sharing these properties. In fact, though we will not make use of this, $A^p((X^\mu)^*)$ is isomorphic to $A^p(X^*)$ and (by fixing such isomorphisms) these maps give the atlas of charts that turns $\T^p(X)$ into a complex Banach manifold modeled on $A^p(X)$. See the proof of \cite[Theorem 3]{GardLak} and \cite[Sections 3 \& 4]{Yanagishita2014} for details on this manifold perspective. 

Next, Ahlfors-Weill \cite{AW62} (see \cite[Section 6.3, Lemma 6]{GardLak} for the statement we use) provides that when one further restricts to the open ball of radius $\frac{1}{2}$ in $A^p((X^\mu)^*)$, the map that sends $\phi \mapsto \mu[\phi]$ is an inverse to the Bers embedding. Therefore, for each $[\mu] \in \T^p(X)$, by shrinking the neighborhoods $\mathcal{U}([\mu])$, we have a neighborhood $\mathcal U_h([\mu])$ such that for each $[\nu]\in \mathcal U_h([\mu])$, $R_{f^\mu}([\nu]) = [\nu\circ (f^{\mu})^{-1}]$ maps under the Bers embedding into the open ball of radius $\frac{1}{2}$ in $A^p((X^{\mu})^*)$. Thus $[\nu\circ (f^{\mu})^{-1}]$ has a harmonic representative, $\nu_h \in \CL^{p,\infty}(X^\mu)$, and the discussion preceding the lemma statement and \Cref{prop:QDLp} shows that $f^{\nu_h}:X^{\mu}\to X^{\nu}$ is asymptotically conformal.

Finally, let $[\mu], [\mu'] \in \T^p(X)$. Consider a continuous path $c:[0,1] \to \T^p(X)$ such that $c(0) = [\mu]$ and $c(1) = [\mu']$, and write $[\mu_t]= c(t)$.  The set $\{\mathcal  U_h([\mu_t]): t \in [0,1]\}$ is an open cover of $c[0,1]$; by compactness, there is a partition $0 =t_1 < t_2 < \ldots < t_{n+1} = 1$, and a finite subcover $\{\mathcal  U_h([\mu_{t_1}]), \ldots, \mathcal U_h([\mu_{t_n})]\}$ such that $c[t_i,t_{i+1}] \subset \mathcal U_h([\mu_{t_i}])$ for each $i$. Now take harmonic Beltrami differentials $\nu_i \in [\mu_{t_{i+1}} \circ (f^{\mu})^{-1}]$ as above, for $i=1, \ldots, n$. By the previous paragraph, $f^{\nu_i}: X^{\mu_{t_i}} \to X^{\mu_{t_{i+1}}}$ is asymptotically conformal. Therefore $f^{\nu_n} \circ f^{\nu_{n-1}} \circ \ldots \circ f^{\nu_1} \colon X^\mu \to X^{\mu'}$ is isotopic to $f \colon X^ \mu \to X^{\mu'}$, and is asymptotically conformal.

\end{proof}

\begin{corollary}\label{cor:lpimpliesac}
    If $X$ satisfies Lehner's condition, then $\T^{p}(X) \subset \T_{0}(X)$ and $\Mod^{p}(X) < \Mod_{0}(X)$. 
\end{corollary}

Combining this with \Cref{cor:rigid finite index} and \Cref{thm:mainbddpieces} we immediately obtain the proof of \Cref{cor:mainlp}. We note that $X$ having bounded pieces implies that $X$ satisfies Lehner's condition by \Cref{lem:bddpiecesbddgeometry}.

\begin{proof}[Proof of \Cref{cor:mainlp}]
    Our assumption readily yields  $\SHo(S) < \Mod^{p}(X)$, and \Cref{cor:lpimpliesac} gives that $\Mod^p(X) < \Mod_{0}(X)$. Thus we can apply \Cref{thm:mainbddpieces}.
\end{proof}

\section{Exotic examples}\label{sec:examples}
In this section we fix a rigid structure $\cR$ on $S$ and construct a variety of hyperbolic Riemann surface structures $X$ on $S$ exhibiting various qualities, for which $\SHo(S)$ acts.  We describe essentially two different constructions.  The first, carried out in \Cref{S:general twisting construction}, involves twist deformations.  This construction produces examples where the intrinsic geometry of pieces are both bounded and unbounded, while still maintaining bounded pieces.  In particular, \Cref{thm:mainbddpieces} applies so that $\SHo(S)$ has finite index in $\Mod_0(X)$.

The second construction is described in \Cref{S:unbounded pieces}, and produces examples of structures $X$ with unbounded pieces, while $\SHo(S)$ still acts asymptotically conformally.  For the examples, \Cref{thm:mainbddpieces} does not apply, and we do not know whether $\SHo(S)$ has finite index in $\Mod_0(X)$; see \Cref{Q:unbounded pieces}.

All constructions have a similar setup which we now describe.  Start with a hyperbolic Riemann surface structure $X_0$ on $S$ that is $\cR$--compatible; we will shortly impose some additional constraints on $X_0$. (Recall that this means we have an implicit homeomorphism $S \cong X$.)  We will construct a new structure as a deformation $f \colon X_0 \to X$.  These are typically {\em not} quasiconformal deformations in our constructions.

Enumerate the ends $e_1,\ldots,e_n$ and choose compatible neighborhoods $U_j = Y_1^j \cup Y_2^j \cup \cdots$ for each $e_j$. For every $j=2,\ldots,n$, we choose a particular element of $\SHo(S)$ interchanging $U_1$ and $U_j$ by a rigid map on their union, and acting as the identity on all other $U_i$.  
Label this finite collection of elements of $\SHo(S)$ as $\Omega$. We will describe the new structure $f \colon X_0 \to X$ by describing it on $U=U_1$ and use the elements of $\Omega$ to push it forward to the other components (and then extend over the remaining compact part in any way we like).  We also write $Y_k = Y_k^1$ for all $k \geq 1$.

We fix an element $F \in \PSHo(S)$ that acts on $U$ by a rigid map sending $Y_k$ to $Y_{k+1}$, for all $k \geq 1$.  We additionally assume that $F$ is the identity on $U_j$ for all $j=3,\ldots,n$, and on $U_2$ we assume that $F$ restricts to the conjugate of $F^{-1}|_{U}$, where we conjugate by the chosen element in $\Omega \subset \SHo(S)$ that interchanged $U$ and $U_2$ above.
We note that since $U_1,\ldots,U_n \subset X_0$ are compatible, we can assume that $F$ is quasiconformal and asymptotically isometric. 

\begin{lemma} \label{L:just check F}
    Suppose that $f \colon X_0 \to X$ and $F \in \SHo(S)$ are as above. Then we have $\SHo(S) < \Mod_0(X)$ if and only if $F \in \Mod_0(X)$.  Similarly, for any $p \geq 1$, $\SHo(S) < \Mod^p(X)$ if and only if $F \in \Mod^p(X)$. 
\end{lemma}

\begin{proof}    
    Since $f:X_0 \rightarrow X$ is defined via pushing forward the hyperbolic structure on $U$ to each other choice of compatible neighborhood of the ends of $X_0$, we have that the set $\Omega \subset \SHo(S)$ is contained in both $\Mod_0(X)$ and $\Mod^p(X)$. Similarly, the set of compactly supported mapping classes is always contained in $\Mod_0(X)$ and $\Mod^{p}(X)$. The desired statement follows from the additional fact $\SHo(S)$ is generated by $F$, $\Omega$, and the set of compactly supported mapping classes. 
\end{proof}

Because the behavior of $F$ on $U_j$ for $j =2,\ldots,n$ is either the identity or conjugate to $F^{-1}|_U$ by an isometry for the structure $X$, the following is immediate.

\begin{lemma} \label{L:just check F|}
    With the assumptions above, $F \in \Mod_0(X)$ if and only if $F|_U \colon U \to F(U)$ is asymptotically conformal.  Likewise, $F \in \Mod^p(X)$ if and only if $\mu_{F|_U}$ is in $\mathcal L^{p,\infty}(U)$. \qed
\end{lemma}

When we need to be more precise about how $F$ acts on $U$ in the structure $X$, we note that the action of $F$ on $X$ is given by $f \circ F \circ f^{-1}$, up to isotopy.

\subsection{General twisting construction} \label{S:general twisting construction}

Our first construction will involve both topological and numerical parameters that can be tuned to produce different features, as we explain below.
The topological parameter is an (infinite) multicurve
\[ \beta = \beta_1 \cup \beta_2 \cup \beta_3 \cup \ldots \subset U\]
with the property that 
 $F(\beta_k) = \beta_{k+1}$ for all $k \geq 1$
We note that since $U \subset X_0$ is compatible and $F$ is rigid on $U$, the restriction $F|_U \colon U \to U$ is an isometric embedding.  Consequently, the lengths of all $\beta_k$ are equal, and hence are bounded above (and below) by some constant.  We assume our structure $X_0$ is chosen so that the length of each $\beta_k$ is $1$. 

The new structure, $f \colon X_0 \to X$, is obtained by a twist deformation on $\beta$.  Specifically, given $\epsilon_k \in \R$, by an {\em $\epsilon_k$-twist} on $\beta_k$ we will mean the distance $\epsilon_k$ twist-deformation on $\beta_k$ to the right if $\epsilon_k > 0$, and the distance $|\epsilon_k|$ twist-deformation to the left if $\epsilon_k < 0$.  If $\epsilon_k = 0$, we perform no twisting.

To state the next lemma, we will need an associated sequence, which is the consecutive differences: $\delta_1 = \epsilon_1$, and $\delta_k = \epsilon_k - \epsilon_{k-1}$, for all $k \geq 2$.

\begin{lemma} \label{lem:twist examples}
    Fix $\beta$ and $\{\epsilon_k\}$, let $f \colon X_0 \to X$ be the associated deformation, and let $\{\delta_k\}$ be the associated sequence of differences, all as above.
    Then:
    \begin{enumerate}
    \item \label{item:QC} $f \colon X_0 \to X$ is quasi-conformal if and only if $\{\epsilon_k\}$ is bounded,
    \item \label{item:AC} $f \colon X_0 \to X$ is asymptotically conformal if and only if $\epsilon_k \to 0$ as $k \to \infty$,
    \item \label{item:Mod_0} $\SHo(S) < \Mod_0(X)$ if and only if $\delta_k \to 0$ as $k \to \infty$, and 
    \item \label{item:Mod^p} for $p \geq 1$, $\SHo(S) < \Mod^p(X)$ if $\displaystyle{\sum_k |\delta_k|^p < \infty}$. 
    \end{enumerate}
\end{lemma}
\begin{proof} 
    We assume that each $\beta_k$ is realized by its geodesic representative in $X_0$.  Since the lengths of all $\beta_k$ are equal to $1$, there is some $C>0$ so that the $C$--neighborhood of $\beta$ is the union of pairwise disjoint annuli; we write $N_k$ for the $C$--neighborhood of $\beta_k$.  It is a straightforward computation to show that if $\{\epsilon_k\}$ is bounded (respectively, $\epsilon_k \to 0$), then $f$ is quasi-conformal (respectively, is asymptotically conformal). This proves half of items \eqref{item:QC} and \eqref{item:AC}. 
    
    Next, because $F(\beta_k) = \beta_{k+1}$ and $F$ is rigid on $U$, it follows that there is a sequence of curves $\alpha_k$ so that $F(\alpha_k) = \alpha_{k+1}$ and $i(\alpha_k,\beta_j) = 0$ for all $j$ and all $k \neq j$ sufficiently large, and $i(\alpha_k,\beta_k) \neq 0$.  Since $F|_U$ is an isometric embedding with respect to $X_0$, the lengths of $\alpha_k$ in $X_0$ are all equal for all $k$ sufficiently large.  If $\{\epsilon_k\}$ is unbounded, then there exists a subsequence of $\{\alpha_k\}$ (corresponding to a subsequence of $\{\epsilon_k\}$ that tends to $\pm \infty$) for which $\ell_X(\alpha_k) \to \infty$, and hence by \Cref{lem:wolpert}, $f$ is not quasiconformal.
    Similarly, if $\epsilon_k$ does not converge to zero, then (after replacing each $\alpha_k$ with its image by a Dehn twist in $\beta_k$ if necessary), there is a subsequence so that $\ell_{X_0}(\alpha_k)/\ell_X(\alpha_k)$ does not limit to $0$, and hence $f$ is not asymptotically conformal by \Cref{cor:nearly isometric}.  This proves the required reverse implications in items \eqref{item:QC} and \eqref{item:AC}.

    Next we consider the sufficiency of the condition on $\{\delta_k\}$ in items \eqref{item:Mod_0} and \eqref{item:Mod^p} simultaneously.  First, recall that by \Cref{L:just check F}, it suffices to show that $F \in \Mod_0(X)$ and $F \in \Mod^p(X)$, respectively, under the given hypotheses.  Furthermore, this containment is entirely dictated by the behavior of $F|_U$, according to \Cref{L:just check F|}.

    The pairwise disjoint collar neighborhoods $\{N_k\}$ of the respective $\{\beta_k\}$ are isometric to each other, and are isometric to the quotient of a $C$--neighborhood of a geodesic in $\hyp$.  We explicitly describe this in the {\em strip model} of $\hyp$, given by
\[ \left( \{z \in \mathbb C \mid 0 < \Imag(z) < \pi \},\frac{|dz|}{\sin(\Imag(z))} \right). \]
In this model, we assume that the geodesic of interest is the one defined by $\Imag(z) = \frac{\pi}2$, and thus there exists $\tau > 0$ so that each $N_k$ is isometric to the hyperbolic annulus $N = V/(z \mapsto z+1)$, where
\[ V = \{z \in \mathbb C \mid \tfrac{\pi-\tau}2 \leq \Imag(z) \leq \tfrac{\pi + \tau}2 \}.\]
We assume the isometries $\sigma_k \colon N \to N_k$ are chosen so that $\sigma_{k+1}^{-1} \circ F|_{N_k} \circ \sigma_k$ is the identity on $N$.

Next, we choose the homeomorphism $f \colon X_0 \to X$ so that it is an isometry on $\displaystyle{U - \bigcup_k N_k}$, and sends each $N_k$ to a collar $N_k'=f(N_k)$ of the image geodesic $f(\beta_k)$.  We note that for each $k \geq 1$, $N_k'$ is also isometric to $N$.  

Next, we choose $f|_{N_k} \colon N_k \to N_k'$ and isometries $\sigma_k' \colon N \to N_k'$ so that for each $k$ the composition,
\[(\sigma_k')^{-1} \circ f \circ \sigma_k \]
lifts to a map $V \to V$ which is the restriction of an affine map with derivative
\[ \left( \begin{array}{cc} 1 & \epsilon_k/\tau \\ 0 & 1\\
\end{array} \right).\]
In fact, the $\mathbb R$--linear map defined by this matrix sends $V$ to itself, and descends to a map $N \to N$ which affects an $\epsilon_k$--twist on the core geodesic, hence can be used to define the required $\sigma_k'$ and $f|_{N_k}$.

By composing the maps above, we can compute that
\[ (\sigma_{k+1}')^{-1} \circ (f \circ F \circ f^{-1})|_{N_k'} \circ \sigma_k',\]
also lifts to a map $V \to V$ which is the restriction of an affine map.  This time, the derivative is
\[ \left( \begin{array}{cc} 1 & \delta_k/\tau \\ 0 & 1\\
\end{array} \right).\]
since $\delta_k$ measures the difference in twisting between $\beta_k$ and $\beta_{k+1}$.
If $\nu$ is the Beltrami coefficient of $f \circ F \circ f^{-1}$, then we can bound the absolute value of its restriction $\nu_k = \nu_{F|_{N_k'}}$ to $N_k'$ using this affine map,
\[ |\nu|_k| \leq |\delta_k/\tau|.\]
Since $F|_U$ is conformal outside $\bigcup_k N_k'$, it follows that if $\delta_k \to 0$ as $k \to \infty$, then $F$ is asymptotically conformal, proving item \eqref{item:Mod_0}.

If we let $A$ denote the hyperbolic area of $N$ (and hence of each $N_k'$), we have the following estimate
\begin{eqnarray*}
\int_U |\nu|^p dA_{\hyp}  & = & \sum_{k=1}^\infty \int_{N_k} |\nu|^p dA_{\hyp} \\
& \leq & \sum_{k=1}^\infty |\delta_k|^pA/\tau^p\\
& = & A/\tau^p\sum_{k=1}^\infty |\delta_k|^p.
\end{eqnarray*}
The finiteness of this final sum is precisely the hypothesis in item \eqref{item:Mod^p}, proving that statement as well.

Finally, the reverse implication of \eqref{item:Mod_0} follows from a similar argument to the reverse implication of \eqref{item:AC}, by considering appropriate transverse curves $\alpha_k$ to $f(\beta_k)$ and the ratio of their lengths to their images by $f \circ F \circ f^{-1}$, then appealing to \Cref{lem:wolpert}.
\end{proof}

Suppose that the multicurve $\beta$ is the union of gluing curves between consecutive pieces, $Y_k$ and $Y_{k+1}$, for all $k \geq 1$.  See Figure~\ref{Fig:Different betas}, surface I, and let $X$ be the associated hyperbolic surface for any choice of $\{ \epsilon_k \}$.  In this case, each of the pieces have bounded geometry, independent of the choice of $\{\epsilon_k\}$.  In fact, $F|_{Y_k}: Y_k \to Y_{k+1}$ is isotopic to an isometry.  At first glance it may seem that this means that $X$ is compatible.  However, the isotopy is not fixed on the boundary, and consequently after the isotopy, the map from $Y_k$ to $Y_{k+1}$ is likely {\em not} rigid, i.e.~given by $\phi_{k+1}^{-1} \circ \phi_k$ as in \Cref{ssec:shogrps}.
On the other hand, \Cref{lem:twist examples} \eqref{item:Mod_0} and \eqref{item:Mod^p} ensure that $\SHo(S) < \Mod_0(X)$ or $\SHo(S) < \Mod^p(X)$, respectively, under the stated hypotheses on $\{\delta_k\}$.

Next, consider a multicurve $\beta$ so that $\beta_k$ is contained in $Y_k$ for all $k$, and $F(\beta_k) = \beta_{k+1}$.  See \Cref{Fig:Different betas}, surface II.  In this case, the hyperbolic structures on the pieces still have bounded geometry, but $F$ is no longer an isometry from $Y_k$ to $Y_{k+1}$, and if $\epsilon_k \to \infty$, while $\delta_k \to 0$, one obtains structures $X$ which are not quasiconformal deformations of $X_0$, but for which $\SHo(S) < \Mod_0(X)$.  If we let $\epsilon_k = \sum_{j=1}^k 1/j$, then $\delta_k = 1/k$, and hence $\SHo(S) < \Mod^p(X)$ for all $p > 1$, yet $X$ is not a quasiconformal deformation of $X_0$.

Finally, we consider a multicurve $\beta$ for which $\beta_k$ intersects both $Y_k$ and $Y_{k+1}$ nontrivially.  See \Cref{Fig:Different betas}, surface III.  Choosing $\epsilon_k \to \infty$ while $\delta_k \to 0$ (or $\delta_k$ $p$--summable for some $p$) gives examples with bounded pieces, but with {\em unbounded} geometry for which $\SHo(S) < \Mod_0(X)$ (respectively, $\SHo(S) < \Mod^p(X)$); indeed, the boundary of $Y_k$ in $X$ will have length tending to infinity with $k$.

\def\arr{-{Stealth[length=3mm, width=2mm]}}
\begin{figure}[htb]
\begin{center}
\begin{tikzpicture}
    \node at (0,0) {\includegraphics[width=9cm]{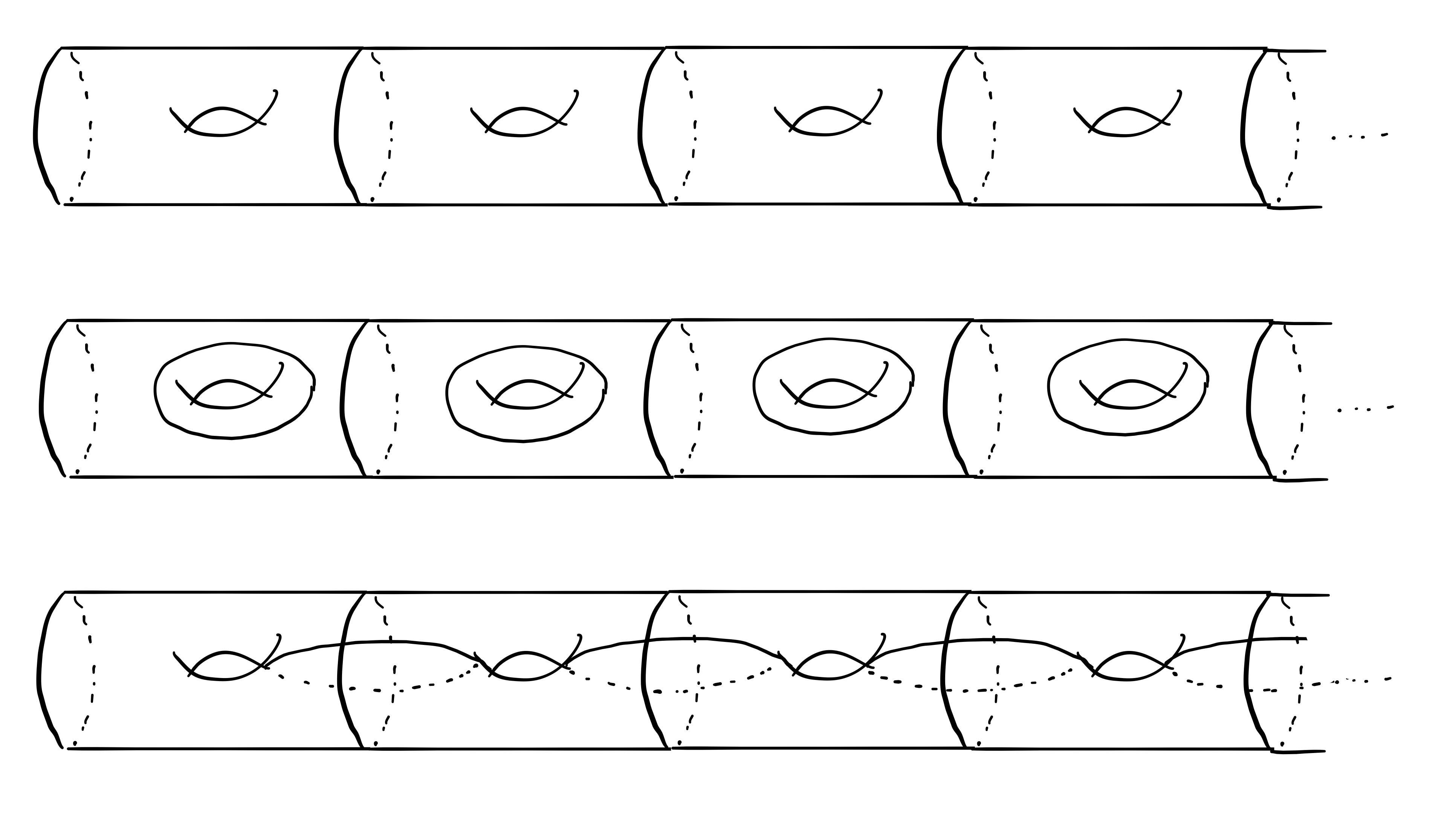}};
    \node at (-2.2,2.55) {\small $\beta_1$};
    \node at (-.3,2.55) {\small $\beta_2$};
    \node at (1.5,2.55) {\small $\beta_3$};
    \node at (3.3,2.55) {\small $\beta_4$};
    \node at (-3.6,-.1) {\small $\beta_1$};
    \node at (-1.85,-.1) {\small $\beta_2$};
    \node at (.1,-.1) {\small $\beta_3$};
    \node at (1.9,-.1) {\small $\beta_4$};
    \node at (-2.8,-1.8) {\small $\beta_1$};
    \node at (-.9,-1.8) {\small $\beta_2$};
    \node at (1,-1.8) {\small $\beta_3$};
    \node at (2.9,-1.8) {\small $\beta_4$};
    \node at (-5,1.8) {\Large I.};
    \node at (-5,.1) {\Large II.};
    \node at (-5,-1.6) {\Large III.};
\end{tikzpicture}
\caption{Three examples of the multicurve $\beta$, producing different behaviors.} \label{Fig:Different betas}
\end{center}
\end{figure}
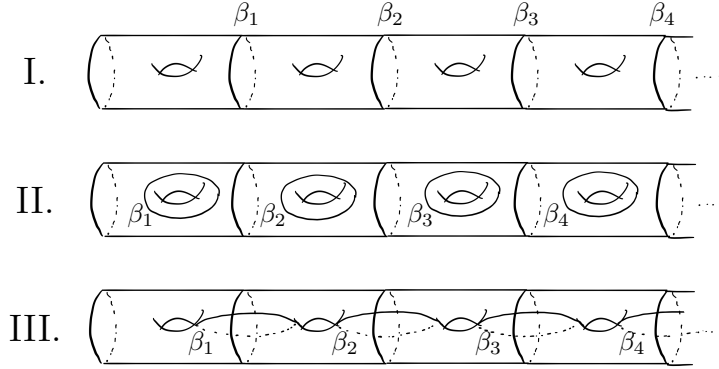

\subsection{Unbounded pieces}\label{S:unbounded pieces}

For our second construction, we build a hyperbolic structure on $U$ directly and show that $F|_U \colon U \to U$ is asymptotically conformal (rather than viewing it as a deformation of a fixed structure $X_0$).

The basic building blocks for this construction are topologically annuli, geometrically realized as hyperbolic surfaces with geodesic boundary and corners; See \Cref{Fig:funny annulus}.  Specifically, for any $.9 < t \leq 1$, we consider the annulus $A_t$ with one boundary component a closed geodesic, and the other a concatenation of four geodesic segments meeting at right angles.  The four segments have lengths $2,2,2,2t$. That such a hyperbolic structure exists follows from explicit computation using hyperbolic trigonometry.

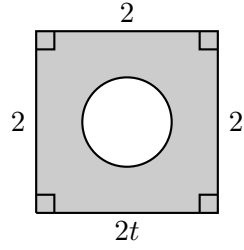
\begin{figure}
\begin{center}
\begin{tikzpicture}
\begin{scope}[scale = 1.2] 
    \filldraw[opacity = .2] (0,0) -- (2,0) -- (2,2) -- (0,2) -- (0,0);
    \draw[thick] (0,0) -- (2,0) -- (2,2) -- (0,2) -- (0,0);
    \filldraw[white] (1,1) circle (14pt);
    \draw[thick] (.2,0) -- (.2,.2) -- (0,.2);
    \draw[thick] (1.8,2) -- (1.8,1.8) -- (2,1.8);
    \draw[thick] (0,1.8) -- (.2,1.8) -- (.2,2);
    \draw[thick] (1.8,0) -- (1.8,.2) -- (2,.2);
    \draw[thick] (1,1) circle (14pt);
    \node at (1,-.2) {$2t$};
    \node at (1,2.2) {$2$};
    \node at (-.2,1) {$2$};
    \node at (2.2,1) {$2$};
\end{scope}
\end{tikzpicture}
\caption{The basic building block} \label{Fig:funny annulus}
\end{center}
\end{figure}

Next, we consider a sequence $\{t_k\}_{k=1}^\infty$ with $.9 < t_k \leq 1$ and $t_k \to 1$ as $k \to \infty$.  We construct a hyperbolic structure on $U$ from the sequence $\{A_{t_k}\}_{k=1}^\infty$ by first gluing the ``right side" of $A_{t_k}$ to the ``left side" of $A_{t_{k+1}}$ for all $k \geq 1$, resulting in a hyperbolic surface-with-corners which is {\em homeomorphic} to a half-infinite strip, $[0,\infty) \times [0,1]$, minus infinitely many disks of Euclidean radius $\frac14$ centered on $\N \times \{\frac12\}$; see \Cref{Fig:gluing funny annuli}.  We next choose a sequence of hyperbolic one-holed tori with geodesic boundary, $\{\Sigma_k\}$ whose hyperbolic structures converge in $\T(S_{1,1})$ and whose lengths match the lengths of the geodesic boundary components.  Finally, we glue the top of the strip to the bottom by isometry to produce a hyperbolic surface $W$ with infinitely many closed geodesic boundary components, and produce $U$ by gluing $\Sigma_k$ to the $k^{th}$ boundary component of $W$ by isometry (the $0^{th}$ boundary component remains exposed); See \Cref{Fig:Unbounded pieces}.

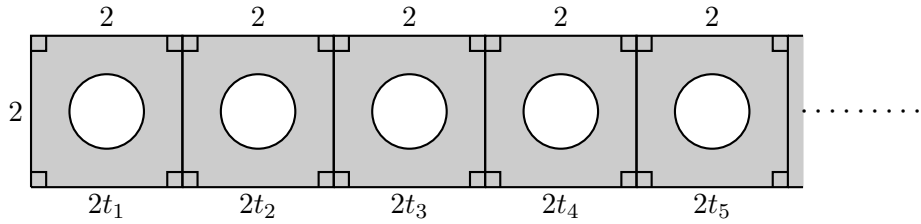
\begin{figure}
\begin{center}
\begin{tikzpicture}
\begin{scope}[scale = 1] 
    \filldraw[opacity = .2] (0,0) -- (2,0) -- (2,2) -- (0,2) -- (0,0);
    \draw[thick] (0,0) -- (2,0) -- (2,2) -- (0,2) -- (0,0);
    \filldraw[white] (1,1) circle (14pt);
    \draw[thick] (.2,0) -- (.2,.2) -- (0,.2);
    \draw[thick] (1.8,2) -- (1.8,1.8) -- (2,1.8);
    \draw[thick] (0,1.8) -- (.2,1.8) -- (.2,2);
    \draw[thick] (1.8,0) -- (1.8,.2) -- (2,.2);
    \draw[thick] (1,1) circle (14pt);
    \node at (1,-.25) {$2t_1$};
    \node at (1,2.25) {$2$};
    \node at (-.2,1) {$2$};
\end{scope}
\begin{scope}[shift = {(2,0)}] 
    \filldraw[opacity = .2] (0,0) -- (2,0) -- (2,2) -- (0,2) -- (0,0);
    \draw[thick] (0,0) -- (2,0) -- (2,2) -- (0,2) -- (0,0);
    \filldraw[white] (1,1) circle (14pt);
    \draw[thick] (.2,0) -- (.2,.2) -- (0,.2);
    \draw[thick] (1.8,2) -- (1.8,1.8) -- (2,1.8);
    \draw[thick] (0,1.8) -- (.2,1.8) -- (.2,2);
    \draw[thick] (1.8,0) -- (1.8,.2) -- (2,.2);
    \draw[thick] (1,1) circle (14pt);
    \node at (1,-.25) {$2t_2$};
    \node at (1,2.25) {$2$};
\end{scope}
\begin{scope}[shift = {(4,0)}] 
    \filldraw[opacity = .2] (0,0) -- (2,0) -- (2,2) -- (0,2) -- (0,0);
    \draw[thick] (0,0) -- (2,0) -- (2,2) -- (0,2) -- (0,0);
    \filldraw[white] (1,1) circle (14pt);
    \draw[thick] (.2,0) -- (.2,.2) -- (0,.2);
    \draw[thick] (1.8,2) -- (1.8,1.8) -- (2,1.8);
    \draw[thick] (0,1.8) -- (.2,1.8) -- (.2,2);
    \draw[thick] (1.8,0) -- (1.8,.2) -- (2,.2);
    \draw[thick] (1,1) circle (14pt);
    \node at (1,-.25) {$2t_3$};
    \node at (1,2.25) {$2$};
\end{scope}
\begin{scope}[shift = {(6,0)}] 
    \filldraw[opacity = .2] (0,0) -- (2,0) -- (2,2) -- (0,2) -- (0,0);
    \draw[thick] (0,0) -- (2,0) -- (2,2) -- (0,2) -- (0,0);
    \filldraw[white] (1,1) circle (14pt);
    \draw[thick] (.2,0) -- (.2,.2) -- (0,.2);
    \draw[thick] (1.8,2) -- (1.8,1.8) -- (2,1.8);
    \draw[thick] (0,1.8) -- (.2,1.8) -- (.2,2);
    \draw[thick] (1.8,0) -- (1.8,.2) -- (2,.2);
    \draw[thick] (1,1) circle (14pt);
    \node at (1,-.25) {$2t_4$};
    \node at (1,2.25) {$2$};
\end{scope}
\begin{scope}[shift = {(8,0)}] 
    \filldraw[opacity = .2] (0,0) -- (2.2,0) -- (2.2,2) -- (0,2) -- (0,0);
    \draw[thick] (0,0) -- (2,0) -- (2,2) -- (0,2) -- (0,0);
    \draw[thick] (2,0) -- (2.2,0);
    \draw[thick] (2,2) -- (2.2,2);
    \filldraw[white] (1,1) circle (14pt);
    \draw[thick] (.2,0) -- (.2,.2) -- (0,.2);
    \draw[thick] (1.8,2) -- (1.8,1.8) -- (2,1.8);
    \draw[thick] (0,1.8) -- (.2,1.8) -- (.2,2);
    \draw[thick] (1.8,0) -- (1.8,.2) -- (2,.2);
    \draw[thick] (1,1) circle (14pt);
    \node at (1,-.25) {$2t_5$};
    \node at (1,2.25) {$2$};
\end{scope}
\node at (11,1) {\large $\cdots \cdots \cdots $};
\end{tikzpicture}
\caption{Hyperbolic structure on $U$.} \label{Fig:gluing funny annuli}
\end{center}
\end{figure}

\begin{figure}[htb]
\begin{center}
\begin{tikzpicture}
    \node at (0,0) {\includegraphics[width=12cm]{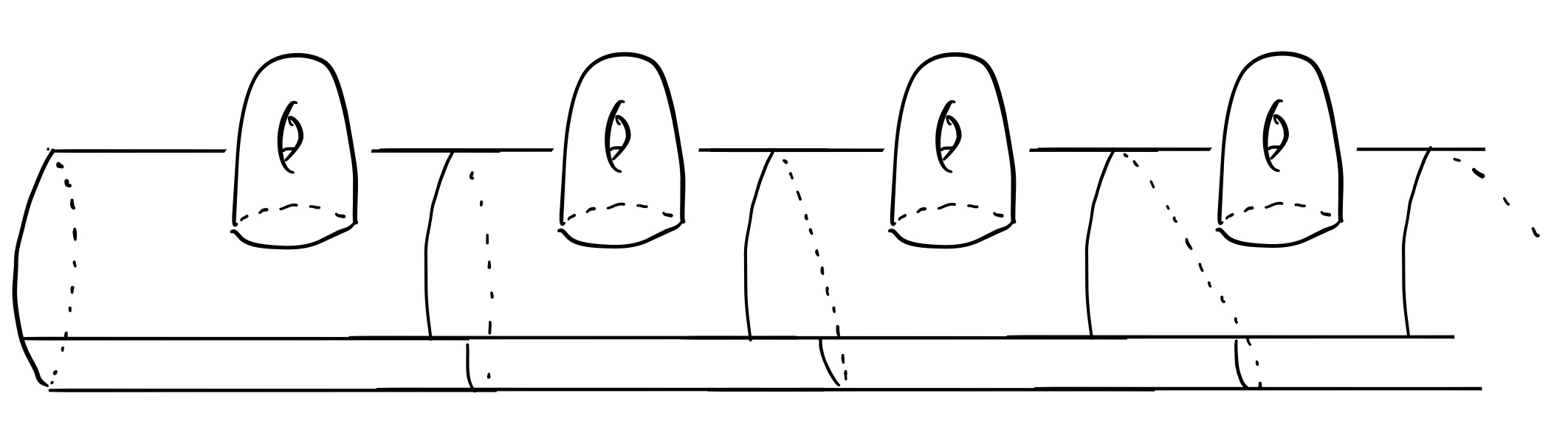}};
    \node at (-4.5,1) {$\Sigma_1$};
    \node at (-2,1) {$\Sigma_2$};
    \node at (.6,1) {$\Sigma_3$};
    \node at (3,1) {$\Sigma_4$};
    \node at (-5,-.5) {$A_{t_1}$};
    \node at (-2,-.5) {$A_{t_2}$};
    \node at (.3,-.5) {$A_{t_3}$};
    \node at (3,-.5) {$A_{t_4}$};
\end{tikzpicture}
\caption{An example with unbounded pieces.} \label{Fig:Unbounded pieces}
\end{center}
\end{figure}
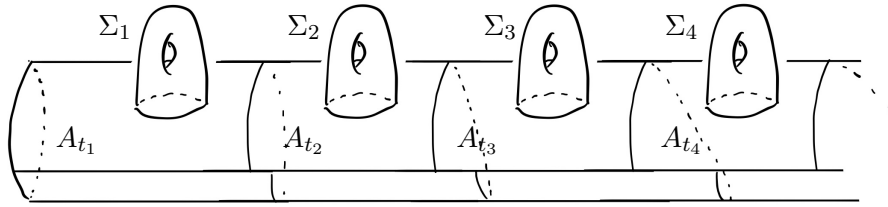

\begin{proposition} \label{prop:example unbounded pieces}
    Given $\{t_k\} \in (.9,1]$ with $t_k \to 1$ and $\{\Sigma_k\}$ as above, we let $X$ be the hyperbolic structure whose ends have neighborhoods that are isometric to surfaces as above.    If the gluings of $\Sigma_k$ to $W$ are chosen appropriately, then $\SHo(S) < \Mod_0(X)$.  Furthermore, if 
    \[\sum_{k=1}^\infty (1-t_k) = \infty,\] then $X$ does not have bounded pieces.
\end{proposition}
\begin{proof}
    Choose a neighborhood $V$ of the ``glued ray" that is standard in the strip model, and construct $F|_U$ so that it preserves $V$ and the restriction to $U \setminus V$ sends $A_{t_k} \setminus V$ to $A_{t_{k+1}} \setminus V$.

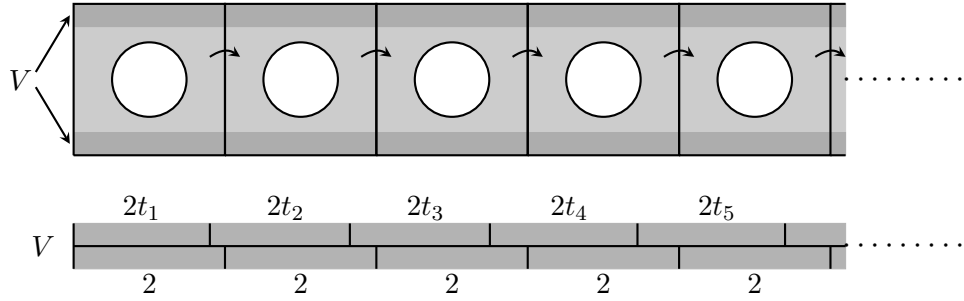
\begin{figure}[h]
\begin{center}
\begin{tikzpicture}
\begin{scope}[scale = 1] 
    \filldraw[opacity = .2] (0,0) -- (2,0) -- (2,2) -- (0,2) -- (0,0);
    \filldraw[opacity = .15] (0,0) -- (2,0) -- (2,.3) -- (0,.3) -- (0,0);
    \filldraw[opacity = .15] (0,1.7) -- (2,1.7) -- (2,2) -- (0,2) -- (0,1.7);
    \draw[thick] (0,0) -- (2,0) -- (2,2) -- (0,2) -- (0,0);
    \filldraw[white] (1,1) circle (14pt);
    \draw[thick, - stealth] (1.8,1.3) .. controls (1.95,1.4) and (2.05,1.4) .. (2.2,1.3);
    \draw[thick] (1,1) circle (14pt);
    \node at (-.7,1) {$V$};
    \draw[thick,- stealth] (-.5,1.1) -- (-.05,1.85);
    \draw[thick,- stealth] (-.5,.9) -- (-.05,.15);
\end{scope}
\begin{scope}[shift = {(2,0)}] 
    \filldraw[opacity = .2] (0,0) -- (2,0) -- (2,2) -- (0,2) -- (0,0);
    \filldraw[opacity = .15] (0,0) -- (2,0) -- (2,.3) -- (0,.3) -- (0,0);
    \filldraw[opacity = .15] (0,1.7) -- (2,1.7) -- (2,2) -- (0,2) -- (0,1.7);
    \draw[thick] (0,0) -- (2,0) -- (2,2) -- (0,2) -- (0,0);
    \filldraw[white] (1,1) circle (14pt);
    \draw[thick, - stealth] (1.8,1.3) .. controls (1.95,1.4) and (2.05,1.4) .. (2.2,1.3);
    \draw[thick] (1,1) circle (14pt);
\end{scope}
\begin{scope}[shift = {(4,0)}] 
    \filldraw[opacity = .2] (0,0) -- (2,0) -- (2,2) -- (0,2) -- (0,0);
    \filldraw[opacity = .15] (0,0) -- (2,0) -- (2,.3) -- (0,.3) -- (0,0);
    \filldraw[opacity = .15] (0,1.7) -- (2,1.7) -- (2,2) -- (0,2) -- (0,1.7);
    \draw[thick] (0,0) -- (2,0) -- (2,2) -- (0,2) -- (0,0);
    \filldraw[white] (1,1) circle (14pt);
    \draw[thick, - stealth] (1.8,1.3) .. controls (1.95,1.4) and (2.05,1.4) .. (2.2,1.3);
    \draw[thick] (1,1) circle (14pt);
\end{scope}
\begin{scope}[shift = {(6,0)}] 
    \filldraw[opacity = .2] (0,0) -- (2,0) -- (2,2) -- (0,2) -- (0,0);
    \filldraw[opacity = .15] (0,0) -- (2,0) -- (2,.3) -- (0,.3) -- (0,0);
    \filldraw[opacity = .15] (0,1.7) -- (2,1.7) -- (2,2) -- (0,2) -- (0,1.7);
    \draw[thick] (0,0) -- (2,0) -- (2,2) -- (0,2) -- (0,0);
    \filldraw[white] (1,1) circle (14pt);
    \draw[thick, - stealth] (1.8,1.3) .. controls (1.95,1.4) and (2.05,1.4) .. (2.2,1.3);
    \draw[thick] (1,1) circle (14pt);
\end{scope}
\begin{scope}[shift = {(8,0)}] 
    \filldraw[opacity = .2] (0,0) -- (2.2,0) -- (2.2,2) -- (0,2) -- (0,0);
    \filldraw[opacity = .15] (0,0) -- (2.2,0) -- (2.2,.3) -- (0,.3) -- (0,0);
    \filldraw[opacity = .15] (0,1.7) -- (2.2,1.7) -- (2.2,2) -- (0,2) -- (0,1.7);
    \draw[thick] (0,0) -- (2,0) -- (2,2) -- (0,2) -- (0,0);
    \draw[thick] (2,0) -- (2.2,0);
    \draw[thick] (2,2) -- (2.2,2);
    \filldraw[white] (1,1) circle (14pt);
    \draw[thick, - stealth] (1.8,1.3) .. controls (1.95,1.4) and (2.05,1.4) .. (2.2,1.3);
    \draw[thick] (1,1) circle (14pt);
\end{scope}
\node at (11,1) {\large $\cdots \cdots \cdots $};
\begin{scope}[shift = {(0,-1.2)}]
\filldraw[opacity = .3] (0,.3) -- (10.2,.3) -- (10.2,-.3) -- (0,-.3) -- (0,.3);
\draw[thick] (0,0) -- (10.2,0);
\draw[thick] (0,.3) -- (0,-.3);
\draw[thick] (2,0) -- (2,-.3);
\draw[thick] (4,0) -- (4,-.3);
\draw[thick] (6,0) -- (6,-.3);
\draw[thick] (8,0) -- (8,-.3);
\draw[thick] (10,0) -- (10,-.3);
\draw[thick] (1.8,0) -- (1.8,.3);
\draw[thick] (3.65,0) -- (3.65,.3);
\draw[thick] (5.5,0) -- (5.5,.3);
\draw[thick] (7.45,0) -- (7.45,.3);
\draw[thick] (9.4,0) -- (9.4,.3);
\node at (1,-.5) {$2$};
\node at (3,-.5) {$2$};
\node at (5,-.5) {$2$};
\node at (7,-.5) {$2$};
\node at (9,-.5) {$2$};
\node at (.9,.5) {$2t_1$};
\node at (2.8,.5) {$2t_2$};
\node at (4.65,.5) {$2t_3$};
\node at (6.55,.5) {$2t_4$};
\node at (8.5,.5) {$2t_5$};
\node at (-.4,0) {$V$};
\node at (11,0) {\large $\cdots \cdots \cdots $};
\end{scope}
\end{tikzpicture}
\caption{Top: Neighborhood $V$ (without $\Sigma_n$ shown) and action of $F$ outside $V$.  Bottom: Neighborhood of $V$ with parameters illustrated.} \label{Fig:Defining F}
\end{center}
\end{figure}

The neighborhood $V$ is isometric to the region $V_0$ in the strip model of $\hyp$ given by
\[ V_0 = \left\{z \in \mathbb C \mid \tfrac{\pi-\tau}2 \leq \mbox{Im}(z) \leq \tfrac{\pi+\tau}2, \, \mbox{and Re}(z) \geq 0 \right\}\]
for some $\tau>0$.
To describe $F|_V$, we construct a conjugate by the isometry $V \to V_0$, which we denote $F_0 \colon V_0 \to V_0$.
We can write this explicitly as follows. 
The top/bottom of $V_0$ is the set of points $z \in \mathbb C$ of the form $z = x + \tfrac{\pi \pm \tau}2i$ for $x \in [0,\infty)$, respectively.  In order for the definitions of $F$ on $\overline{U \setminus V}$ and on $V$ to agree on the boundary, we need to make sure that it sends $A_{t_k} \cap V$ to $A_{t_{k+1}} \cap V$.  Along the bottom, we can thus define $F_0$ by
\[ F_0\left( x+\tfrac{\pi - \tau}2i\right) = x+2 + \tfrac{\pi - \tau}2i. \]
Along the top, the map has a more complicated formula. To describe it, we first subdivide $[0,\infty)$ at the points 
\[ s_0 = 0, \, s_1=2t_1, \, s_2 = 2(t_1+t_2) , \, \ldots , \, s_k = 2(t_1+\ldots+t_k), \ldots  \]
Then, for $x \in [s_{k-1},s_k]$, we define
\[F_0\left( x + \tfrac{\pi+\tau}2i \right)  = \tfrac{t_{k+1}}{t_k}(x-s_{k-1}) + s_k +\tfrac{\pi+\tau}2i.\]
For all $k$, $F_0$ linearly maps intervals to intervals:
\[ \{x+\tfrac{\pi+\tau}2i \mid x \in [s_{k-1},s_k]\} \to \{x+ \tfrac{\pi+\tau}2i \mid x \in [s_k,s_{k+1}]\}.\]
From the formula above, it is not difficult to see that one can extend $F_0$ over $V_0$ by a quasiconformal map so that $\mu_{F_0}(z) \to 0$ as $\mbox{Re}(z) \to \infty$.  For example, we can subdivide $V_0$ into rectangles $[s_{k-1},s_k] \times \left[ \tfrac{\pi-\tau}2,\tfrac{\pi+\tau}2 \right]$, and for all $z = x+iy$ in this rectangle, define $F_0$ linearly on horizontal segments by linearly interpolating between $F_0$ on the bottom and $F_0$ on the top, as already defined.  The derivative of this map for $z$ in this rectangle is given by
\[ \left( \begin{array}{cc} a(z) & b(z) \\ 0 & 1 \end{array} \right), \]
where $a(z) \to 1$ and $b(z) \to 0$ uniformly as $\mbox{Re}(z) \to \infty$ since $t_k \to 1$ as $k \to \infty$ implies $s_k-s_{k-1} \to 2$ and $\frac{t_{k+1}}{t_k} \to 1$ as $k \to \infty$. In particular, $\mu_{F_0}(z) \to 0$ as $\mbox{Re}(z) \to \infty$, as required.

Since $t_k \to 1$ as $k \to \infty$, we also see that $A_{t_k}$ converges to $A_1$ as $k \to \infty$, so that we can construct $F$ from $A_{t_k} \setminus V \to A_{t_{k+1}} \setminus V$ with Beltrami coefficients tending to $0$, and which linearly scales the metric on the boundary geodesics.  Similarly, since $\Sigma_k$ converges in $\T(S_{1,1})$, we can choose $F$ to send $\Sigma_k \to \Sigma_{k+1}$ with Beltrami coefficient tending to $0$ and also scaling the metric on the boundary geodesic.  Now we just ensure that the $\Sigma_k$ are glued to $W$ so the definition of $F$ on each $\Sigma_k$ agrees with it on $A_{t_k}$.
Thus, $F$ is asymptotically quasiconformal.

The statement that divergence of the series implies that $X$ does not have bounded pieces follows from the fact that the partial sums measure the drift between where the top of $A_{t_k}$ is along the ray relative to the bottom.
\end{proof}

In light of the above examples, we have the following natural question: 

\begin{question}
\label{Q:unbounded pieces}
    For the hyperbolic structures on $X$ from Proposition~\ref{prop:example unbounded pieces} with unbounded pieces, is it true that $[\Mod_0(X):\SH(S)] < \infty$?  Can the construction be carried out for some/all $p \geq 1$ to get $\SHo(S) < \Mod^p(X)$ while still having unbounded pieces?
\end{question}

\printbibliography

\end{document}